\documentclass[11pt]{amsart}

\usepackage{amsxtra,amssymb,amsmath,amscd,url,listings,amsrefs, listings, stmaryrd}
\usepackage[utf8]{inputenc}
\usepackage{eucal}
\usepackage{fullpage}
\usepackage{scrtime}
\usepackage{hyperref}
\hypersetup{pdfauthor=Valentin Blomer and Etienne Fouvry and Emmanuel Kowalski and Philippe Michel and Djordje Milicevic, pdftitle=On moments of twisted L-functions}

\renewcommand{\leq}{\leqslant}
\renewcommand{\geq}{\geqslant}
 
\numberwithin{equation}{section}

\newcommand{\uple}[1]{\text{\boldmath${#1}$}}

\def\stacksum#1#2{{\stackrel{{\scriptstyle #1}}
{{\scriptstyle #2}}}}

\newcommand{\bfalpha}{\uple{\alpha}}
\newcommand{\bfbeta}{\uple{\beta}}
\newcommand{\bfb}{\uple{b}}

\newcommand{\ftchi}{f\otimes\chi}

\newcommand{\bfN}{\mathbf{N}}
\newcommand{\Cc}{\mathbf{C}}

\newcommand{\Zz}{\mathbf{Z}}

\newcommand{\Rr}{\mathbf{R}}

\newcommand{\Fq}{{\mathbf{F}_q}}

\newcommand{\Ff}{\mathbf{F}}

\newcommand{\bFq}{\bar{\Ff}_q}

\newcommand{\mcN}{\mathcal{N}}

\newcommand{\mods}[1]{\,(\mathrm{mod}\,{#1})}

\newcommand{\ra}{\rightarrow}

\DeclareMathOperator{\Kl}{\mathrm{Kl}_2}
\DeclareMathOperator{\ET}{\mathrm{ET}}

\newcommand{\eps}{\varepsilon}
\renewcommand{\rho}{\varrho}

\DeclareMathOperator{\GL}{GL}

\DeclareMathSymbol{\gena}{\mathord}{letters}{"3C}
\DeclareMathSymbol{\genb}{\mathord}{letters}{"3E}

\def\multsum{\mathop{\sum\cdots \sum}\limits}

\def\sums{\mathop{\sum \Bigl.^{*}}\limits}

\newcounter{bnd}

\theoremstyle{plain}
\newtheorem{theorem}{Theorem}[section]
\newtheorem*{theorem*}{Theorem}
\newtheorem{lemma}[theorem]{Lemma}
\newtheorem{corollary}[theorem]{Corollary}
\newtheorem{conjecture}[theorem]{Conjecture}

\newtheorem{proposition}[theorem]{Proposition}

\newtheorem{bound}[bnd]{Bound}

\theoremstyle{remark}

\theoremstyle{definition}

\newtheorem{remark}[theorem]{Remark}

\newcommand{\mcB}{\mathcal{B}}

\newcommand{\mfa}{\mathfrak{a}}

\newcommand{\lf}{\lambda_f}
\newcommand{\lamg}{\lambda_g}
\newcommand{\vphi}{\varphi}

\renewcommand{\geq}{\geqslant}
\renewcommand{\leq}{\leqslant}
\renewcommand{\Re}{\mathfrak{Re}\,}
\renewcommand{\Im}{\mathfrak{Im}\,}

\newcommand{\ov}[1]{\overline{#1}}

\newcommand\sumsum{\mathop{\sum\sum}\limits}
\newcommand\sumsumsum{\mathop{\sum\sum\sum}\limits}

\newcommand\rpfree{1/144}

\begin{document}
 
\title{On moments of twisted $L$-functions} 

\author{Valentin Blomer}
\address{Mathematisches Institut, Universit\"at G\"ottingen,
  Bunsenstr. 3-5, 37073 G\"ottingen, Germany} \email{vblomer@math.uni-goettingen.de}

\author{\'Etienne Fouvry}
\address{Laboratoire de Math\'ematiques d'Orsay, Universit\' e Paris--Saclay  \\
    91405 Orsay  \\France}
\email{etienne.fouvry@math.u-psud.fr}

\author{Emmanuel Kowalski}
\address{ETH Z\"urich -- D-MATH\\
  R\"amistrasse 101\\
  CH-8092 Z\"urich\\
  Switzerland} \email{kowalski@math.ethz.ch}

\author{Philippe Michel} \address{EPF Lausanne, Chaire TAN, Station 8, CH-1015
  Lausanne, Switzerland } \email{philippe.michel@epfl.ch}

 \author{Djordje Mili\'cevi\'c}
 \address{Department of Mathematics,
Bryn Mawr College,
101 North Merion Avenue,
Bryn Mawr, PA 19010-2899, U.S.A.}
 \email{dmilicevic@brynmawr.edu}

 \thanks{V.\ B.\ was supported by the ERC (Starting Grant 258713) and the
   Volkswagen Foundation.  \'E.\ F.\ thanks ETH Z\"urich, EPF Lausanne
   and the Institut Universitaire de France for financial
   support. Ph. M. was partially supported by the SNF (grant
   200021-137488) and the ERC (Advanced Research Grant 228304). V.\ B.,
   Ph.\ M.\ and E.\ K.\ were also partially supported by a DFG-SNF lead
   agency program grant (grant 200021L\_153647). D. M. acknowledges
   partial support by the NSA (Grant H98230-14-1-0139), NSF (Grant
   DMS-1503629), and ARC (through Grant DP130100674). The United States
   Government is authorized to reproduce and distribute reprints notwithstanding
   any copyright notation herein.}



\begin{abstract}

We study the average of the product of the central values of two $L$-functions
of modular forms $f$ and $g$ twisted by Dirichlet characters to a large prime modulus $q$. As our
principal tools, we use spectral theory to develop bounds on averages of shifted convolution sums with differences ranging over multiples of $q$, and we use the theory of Deligne and Katz to
prove new bounds on bilinear forms in Kloosterman sums with power savings when both variables are near the square root of $q$. When at least one of the forms $f$ and $g$ is non-cuspidal, we obtain an asymptotic formula for the mixed second moment of twisted $L$-functions with a power saving error term. In particular, when both are non-cuspidal, this gives a significant improvement on M.~Young's asymptotic evaluation of the fourth moment of Dirichlet $L$-functions. In the general case, the asymptotic formula with a power saving is proved  under a conjectural estimate for certain bilinear forms in Kloosterman sums.
\end{abstract}

\maketitle

\setcounter{tocdepth}{1}

\tableofcontents

\section{Introduction}\label{intro}

\subsection{Moments of twisted \texorpdfstring{$L$-functions}{L-functions}}
This paper is motivated by the beautiful work of Mat\-thew Young on the
fourth moment of Dirichlet $L$-functions for prime moduli~\cite{MY}:
for a prime $q>2$, let
\[ M_4(q):=\frac{1}{\varphi^*(q)}\sum_{\substack{\chi\mods q\\ \chi
    \text{ primitive}}}|L(\chi,1/2)|^4, \]
where $\varphi^*(q)=q-2$ is the
number of primitive Dirichlet characters modulo $q$,
and
\[ L(\chi,s)=\sum_{n\geq 1}\frac{\chi(n)}{n^s},\quad \Re s>1, \]
is the Dirichlet $L$-function. Young obtained the asymptotic formula 
\begin{equation}\label{Yo}
M_4(q)=P_4(\log q)+O(q^{-\frac{1}{80}(1-2\theta)+\eps})
\end{equation}
for any $\eps>0$, where $P_4$ is a polynomial of degree four with
leading coefficient $1/(2\pi^2)$, and here and in the following the constant $\theta=7/64$ is
the best known approximation towards the Ramanujan--Petersson
conjecture (due to Kim and Sarnak~\cite{KiSa}).

The fourth moment of Dirichlet $L$-functions is a special case of the
more general second moment
\begin{equation}
\label{SecondMomentDisplayed}
M_{f,g}(q)=\frac{1}{\vphi^*(q)}\sum_{\substack{\chi\mods q\\ \chi \text{ primitive}}}L(f\otimes\chi,1/2)\ov{L(g\otimes\chi,1/2)},
\end{equation}
where $q$ is an integer with $q\not\equiv 2\mods{4}$ (since otherwise
there are no primitive characters modulo $q$), $f$ and $g$ denote two fixed
(holomorphic or non-holomorphic) Hecke eigenforms, not necessarily cuspidal, with respective Hecke
eigenvalues $(\lf(n))_{n\geq 1}$, $(\lambda_g(n))_{n\geq 1}$, and
\[ L(f\otimes\chi,s)=\sum_{n\geq 1}\frac{\lf(n)\chi(n)}{n^s}, \quad L(g\otimes\chi,s)=\sum_{n\geq 1}\frac{\lambda_g(n)\chi(n)}{n^s}\quad  (\Re s>1) \]
denote the associated twisted $L$-functions. Indeed, let $E(z)$ denote the central derivative of the Eisenstein
series $E(z,s)$, i.e.,
\begin{equation}\label{Eisendef}
E(z)=\frac{\partial}{\partial s}\Big|_{s=1/2}E(z,s),\ \text{ with }
E(z,s)=\frac12\sum_{(c,d)=1}\frac{y^s}{|cz+d|^{2s}}.
\end{equation}
This is a Hecke eigenform of level $1$ with Hecke
eigenvalues given by the usual divisor function
\begin{equation}
\label{UsualDivisorFunction}
\lambda_E(n)=d(n)=\sum_{ab=n}1
\end{equation}
(see \cite{IWblueYellow}*{\S 3.4} for instance).  We have  $L(\chi,s)^2=L(E\otimes\chi,s)$,
and therefore
\[ M_4(q)=M_{E,E}(q). \]
\par

Our first main result is a significant improvement of the error term
in the fourth moment of Dirichlet $L$-functions \eqref{Yo}.

\begin{theorem}\label{improvedyoung} Let $q$ be a prime. Then for any
  $\eps>0$, we have
\begin{equation*}\label{Yo2}
M_4(q)=P_4(\log q)+O_\eps(q^{-1/32+\eps}).
\end{equation*}
Moreover, under the Ramanujan--Petersson conjecture the exponent
$1/32$ may be replaced by $1/24$.
\end{theorem}

Our second main result is an asymptotic formula for the ``mixed''  moment $M_{f,E}(q)$. 

\begin{theorem}\label{mixedthm} Let $f$ be a cuspidal Hecke eigenform of level $1$  
  and $E$ the 
  Eisenstein series \eqref{Eisendef}. Let $q$ be a prime
  number. Then for any $\eps>0$, we have
\[ M_{f,E}(q)=\frac{1}{\vphi^*(q)}\sum_{\substack{\chi\mods q\\ \chi
    \text{ {\rm  primitive}} }}L(f\otimes\chi,1/2)\ov{L(\chi,1/2)}^2=
\frac{L(f,1)^2}{\zeta(2)}+O_{f,\eps}(q^{-1/68+\eps}). \]
\end{theorem}

Our final result establishes an asymptotic formula for the moment
$M_{f,g}(q)$, conditionally on a bound for a certain family of
algebraic exponential sums.

\begin{theorem}\label{cuspthm} Let $f,g$ be distinct
  cuspidal Hecke eigenforms of level $1$; if they are either both holomorphic or both Maa{\ss} we assume moreover that their root numbers satisfy $\eps(f)\eps(g)=1$. Let $q$ be a prime number. Assume that \emph{Conjecture
    \ref{completesumconj}} below holds. Then for any $\eps>0$, we have
\[ M_{f,g}(q)=\frac{1}{\vphi^*(q)}\sum_{\substack{\chi\mods q \\
      \chi \text{ {\rm primitive}}}}L(f\otimes\chi,1/2) 
  \overline{L(g\otimes\chi,1/2)}=\frac{2L(f\otimes
    g,1)}{\zeta(2)}+O_{f,g,\eps}(q^{-\rpfree+\eps}), \]
where $L(f\otimes g,1)\not=0$ is the value at $1$ of the
Rankin--Selberg $L$-function of $f$ and $g$. Moreover, 
\[ M_{f,f}(q)=\frac{1}{\vphi^*(q)}\sum_{\substack{\chi\mods q \\ \chi \text{ {\rm  primitive}}}}|L(f\otimes\chi,1/2)|^2=P_{f}(\log q)+O_{f,\eps}(q^{-\rpfree+\eps}), \]
where $P_f(X)$ is an explicit polynomial of degree $1$  with coefficients
independent of $q$ and leading coefficient $2 L(\mathrm{sym}^2f,1)\zeta(2)^{-1}$.  
\end{theorem}
\begin{remark}
(1) If $f$ and $g$ are either both holomorphic or both Maa{\ss} with $\eps(f)\eps(g)=-1$, then $M_{f,g}(q)=0$ for parity reasons (see Remark \ref{epsrem}).

\par

(2)  As a rule of thumb, the asymptotic evaluation of $M_{f,g}(q)$ with a
  good error term gets significantly more challenging as the set
  $\{f,g\}$ contains more cusp forms. (On the other hand, the main
  term in the asymptotic expansion of $M_{f,g}(q)$ gets more
  complicated as the set $\{f,g\}$ contains more Eisenstein series.)
\end{remark}

Asymptotic formulas with a power saving for moments in families of $L$-functions are essential prerequisites for many applications including the techniques of amplification, mollification or resonance.  Evaluation of moments becomes more difficult as the analytic conductor of the family increases relative to its size (for example, if considering moments involving higher powers of $L$-functions). In Theorems \ref{improvedyoung}--\ref{cuspthm}, the family is of size $|\mathcal{F}|\asymp q$ and the analytic conductor is $\asymp_{f,g}q^4\asymp|\mathcal{F}|^4$. As is well-known to experts, this is precisely the critical range of relative sizes at which most current analytic techniques fall just short of producing an asymptotic, and, in the few cases where an asymptotic available in this range, some deep input is typically required.

When $f=g$ and $f$ is cuspidal, the moment \eqref{SecondMomentDisplayed} was studied by
Stefanicki~\cite{St} and Gao, Khan and Ricotta~\cite{GKR}.  In both cases,
however, the error term gives only a saving of (at most) a small power
of $\log q$ over the main term.  

Finally, in relation to our Theorem~\ref{mixedthm} on mixed moments, we note the following asymptotic formula, recently established in \cite{DK}:
\[ \frac{1}{\vphi^*(q)}\sum_{\substack{\chi\mods q\\ \chi
    \text{ {\rm  primitive and even}}}}L(f\otimes\chi,1/2)\ov{L(\chi,1/2)}=
\frac{1}{2}L(f,1)+O_{f,\eps}(q^{-1/64+\eps}). \]

\subsection{Outline of the proof and bilinear forms in Kloosterman sums}\label{sec-outline}

\label{PrincipleSection}

In this section we outline the proof of Theorems~\ref{improvedyoung}--\ref{cuspthm}. This will also be an occasion to describe the two main ingredients of our approach, which are of independent interest: efficient treatment of shifted convolution sums (with particularly long shift variables) using the full power of spectral theory, and estimates of bilinear forms in Kloosterman sums (in particular when both variables are close to the square root of the conductor), which we treat using algebraic geometry.

Let $q$ be a prime and $f,g$ be Hecke eigenforms (cuspidal or equal to $E$) of level one. To simplify the forthcoming discussion, we assume in this section that both $f$ and $g$ satisfy the Ramanujan--Petersson conjecture (which is trivial for $E$ and due to Deligne for holomorphic forms \cite{We})
\begin{equation}\label{delignebound}
|\lf(n)|,|\lambda_g(n)|\leq d(n)	
\end{equation}
for all $n\geq 1$, where $d(n)$ denotes the divisor function.

Using the functional equation of $L(f\otimes \chi,s)L(g\otimes\ov\chi,s)$ (cf.\ \eqref{fcteqn}),  we represent the central values as a converging series
\begin{multline*}
  L(f \otimes\chi,1/2)L(g\otimes\ov\chi,1/2)=\sumsum_{m,n\geq 1}\frac{\lf(m)\lamg(n)\chi(m)\ov\chi(n)}{(mn)^{1/2}}V\Bigl(\frac{mn}{q^2}\Bigr)\\
  +\eps(f,g,\chi)\sumsum_{m,n\geq
    1}\frac{\lf(m)\lamg(n)\ov\chi(m)\chi(n)}{(mn)^{1/2}}V\Bigl(\frac{mn}{q^2}\Bigr)
\end{multline*}
for  some essentially bounded function $V(t)$, which depends on the archimedean local factor $L_\infty(f\otimes
\chi,s)L_\infty(g\otimes \ov\chi,s)$ and which decays rapidly as
$t\geq q^{\eps}$ (for any fixed $\eps>0$). An important feature is that this archimedean local factor and the root number $\eps(f,g,\chi)=\pm
1$ both depend on the character $\chi$ only through its \emph{parity}, i.e., through $\chi(-1)=\pm 1$. 
Therefore it is natural to average separately over even or odd characters, and then the root number $\epsilon(f,g,\chi)$ and the cutoff function $V$ will be constant for all $\chi$ in the average.

The orthogonality of characters with given parity (given by~(\ref{eq-ortho}) below) shows  that theses averages are simple combinations of the quantities
\[ B_{f, g}^{\pm}(q) = \sum_{m \equiv \pm n \mods q} \frac{\lambda_f(m)
  \lambda_g(n)}{(mn)^{1/2}}V\left(\frac{mn}{q^2}\right) -
\frac{1}{\vphi^{\ast}(q)} \sum_{(m n, q) = 1} \frac{\lambda_f(m)
  \lambda_g(n)}{(mn)^{1/2}}V\left(\frac{mn}{q^2}\right). \]
\par
The first main term arises from $B_{f, g}^+(q)$ for $m=n$. Putting  this aside
and applying a partition of unity reduces the problem to the evaluation
of bilinear expressions of the type
\begin{multline*}
  B_{f, g}^{\pm}(M,N) = \frac{1}{(MN)^{1/2}}\sum_{\substack{m\equiv\pm n\mods
      q \\ m \not= n}}
  {\lf(m)\lamg(n)}W_1\Bigl(\frac{m}M\Bigr)W_2\Bigl(\frac{n}N\Bigr)\\
  - \frac{1}{q(MN)^{1/2}} \sum_{m, n} \lambda_f(m) \lambda_g(n)
  W_1\Bigl(\frac{m}M\Bigr)W_2\Bigl(\frac{n}N\Bigr),
\end{multline*}
where $ M,N \geq 1$, $MN\leq q^{2+o(1)}$ and $W_{1}$, $W_2$ are test
functions satisfying \eqref{Wbound} below. 
\par
At this point, an off-diagonal ``main term'' appears in the
non-cuspidal case $f=g=E$. This is rather complicated, but extracting
and estimating this term has been done by Young (see~\cite{MY}). We denote by
$\ET_{f, g}^{\pm}(M,N)$ the remaining part of $B_{f, g}^{\pm}(M,N)$ in all cases
(thus $\ET_{f, g}^{\pm}(M,N)=B_{f, g}^{\pm}(M,N)$ unless $f=g=E$). From this analysis, we know that an asymptotic evaluation of the
twisted moment $M_{f,g}(q)$ with power-saving error term follows as
soon as one proves that
\[ \ET_{f, g}^{\pm}(M,N)\ll q^{-\eta} \]
for some absolute constant $\eta>0$. 

The trivial bound
\[ \ET_{f, g}^{\pm}(M,N)\ll_{\eps} (MN)^{1/2}q^{-1+o(1)}, \]
implies that we may assume that $MN$ is close to $q^2$ when estimating $\ET_{f, g}^{\pm}(M,N)$. 
For simplicity, we assume that $MN=q^{2+o(1)}$ in this outline.

At this point, the analysis depends on the relative ranges of $M$ and $N$. 
There are essentially two cases to consider which are handled by two very different methods.

\subsubsection*{Balanced configuration and the shifted convolution problem}
First, when the sizes of $M$ and $N$ are relatively close to each other, we interpret the
congruence condition as an equality over the integers:
\begin{equation} \label{CongruenceCondition}
0\not=m\mp n\equiv 0\mods q \Leftrightarrow\ m\mp n=qr,\text{ for some
} r\not=0.
\end{equation}
\par
For each $r$, we face an instance of the \emph{shifted convolution
  problem} for Hecke eigenvalues. This problem has now a long history
in a variety of contexts (see \cite{MiParkCity} for an overview). The
most powerful methods known today involve the spectral theory of
automorphic forms and usually depend on bounds towards
the Ramanujan--Petersson conjecture (not solely for $f$ and $g$ but for all automorphic forms of level $1$).
  
The typical estimate one can obtain is
\begin{equation}\label{pointwiseSCP}
\ET_{f,g}^\pm(M,N)\ll \frac{q^{o(1)}}{q^{1/2-\theta}}\left(\frac{M}N+\frac NM\right)^{1/2}.
\end{equation}
(see for instance~\cite{MY}*{Theorem 3.3} when $f=g=E$), where we recall that $\theta = 7/64$ is the best known approximation to the Ramanujan--Petersson conjecture and depends among other things on the automorphy of the symmetric fourth power of $\GL(2)$-automorphic representations \cite{KiSa}. 
This bound is quite satisfactory when $M$ and $N$ are close in the
logarithmic scale, but (as can be expected) it becomes weaker as $M,N$
get apart from each other. In particular, when $MN=q^{2+o(1)}$, the
bound is only non-trivial outside the range
\begin{equation}\label{thetarange}
\max(M,N)\geq q^{3/2-\theta-\delta}
\end{equation}
for $\delta > 0$ fixed as small as we need. 

Ideally (under the Ramanujan--Petersson conjecture), it would remain to handle the range
\begin{equation}\label{RPrange}
\max(M,N)\geq q^{3/2 -\delta},
\end{equation}
which we will eventually be able to do (with a small but fixed $\delta>0$) using an alternative set of methods described below and developed in detail in Section~\ref{exponentialsums}. Unfortunately and despite the fact that the current value of $\theta$ is quite small, these methods do not seem always capable to cover the 
range \eqref{thetarange}: specifically, to prove Theorem \ref{cuspthm}, with some positive exponent in place of $\rpfree$, using the bound \eqref{pointwiseSCP} and the results of Section~\ref{exponentialsums}, one would need to have $\theta<1/40$. In addition to such a result being unavailable as yet, an argument that does not seriously depend of the numerical value of $\theta$ is of interest on its own.
 
 The removal of this dependence on the Ramanujan--Petersson conjecture is
precisely one of the main achievements in \cite{BloMil} when $f$ and $g$ are both cuspidal. We adapt this method, which further refines the shifted convolution sum estimates into the very long shift variable range (in particular by exploiting the average over $r$ in \eqref{CongruenceCondition}), to obtain similar uniform 
estimates also when $f$ or $g$ is the Eisenstein series $E$ (see
Section~\ref{sec-shifted}, Theorem~\ref{shiftthmcusp}).
The resulting bound is as follows:
\begin{bound}
The error term $\ET^{\pm}_{f,g}(M,N)$ satisfies
\begin{equation}\label{boundBlMil}
 \ET^{\pm}_{f,g}(M,N)  \ll q^{o(1)} \left( \frac{1}{q^{1/2}} \Bigl(\frac{M}{N} + \frac{N}{M}\Bigr)^{1/2}+\frac{1}{q} \Bigl(\frac{M}{N} + \frac{N}{M}\Bigr)\right)^{1/2}+q^{-1/2+\theta+o(1)}.
\end{equation}
\end{bound}
The estimate \eqref{boundBlMil}
is weaker than \eqref{pointwiseSCP} for $M=N$, but crucially it is non-trivial in the full range complementary to \eqref{RPrange}, thereby acting essentially as if $\theta=0$. 

Up to this point, there are only minor
differences (e.g., having to do with the main terms) between all cases
of $f$ and $g$. We now explain the second ingredient used to cover
the remaining range \eqref{RPrange}, which eventually requires us to consider different cases separately.

\subsubsection*{Unbalanced configuration and bilinear sums of Kloosterman sums}	
We assume that $N=\max(M,N)$ is the longest variable, with $N\geq
q^{3/2-\delta}$ for some small $\delta>0$. Because it is a long
variable, we may gain by applying to it the Voronoi summation formula (followed by a smooth partition of unity).  This leads to a decomposition of $\ET_{f, g}^{\pm}(M,N)$ (up to possible main terms that are dealt with separately) into sums of the type
\[ C^{\pm}(M,N')=\frac{1}{(qMN')^{1/2}} \sumsum_{m,\
  n}\lf(m)\lamg(n)\Kl(\pm mn;q) W_1\left(\frac{m}{M}\right)\widetilde
W_2\left(\frac{n}{N'}\right), \]
where the ``dual'' length $N'$ satisfies
\[ N'\leq N^*:=q^2/N \]
and $\widetilde W_2$ is another smooth function satisfying
\eqref{Wbound}. (Here $\Kl(a;q)$ denotes the Kloosterman sum modulo $q$, normalized 
 so that $|\Kl(a;q)|\leq 2$ by Weil's bound, see \eqref{Kloostermandefined}.)
Thus the goal is now to prove that
\[ C^{\pm}(M,N')\ll q^{-\eta} \]
for some absolute constant $\eta>0$. It turns out that the main
difficulty is when
\[ N'=N^*=q^2/N=q^{o(1)}M, \]
which we now assume.

Such sums are very special cases of bilinear sums in Kloosterman sums
\begin{equation}\label{Kloosbil}
B(\Kl,\bfalpha_U,\bfbeta_V)=\sumsum_{u\leq U,\ v\leq V}\alpha_u\beta_v\Kl(auv;q),
\end{equation}
for $(a,q)=1$, $U,V\leq q$ and some complex numbers $(\alpha_u)_{u\leq U}$, $(\beta_v)_{v\leq V}$. 
The ``trivial'' bound (which follows from Weil's bound) is
\[ \sumsum_{u\leq U,\ v\leq V}\alpha_u\beta_v\Kl(auv;q)\ll \|\alpha\|_2\|\beta\|_2(UV)^{1/2}, \]
and a natural question is whether one can improve that bound at least for suitable values of the parameters $U,V$ and/or $(\alpha_u)_u$, $(\beta_v)_v$.

Specialized to our current situation (taking $U=M$, $V=N^*=q^{o(1)}M$), the trivial bound yields
\[ C^{\pm}(M,N^*)=\frac{1}{(qMN^*)^{1/2}}B(\Kl,\bfalpha_M,\bfbeta_{N^*})\ll q^{o(1)}(MN^*/q)^{1/2}= q^{3/2+o(1)}/N, \]
which is satisfactory as soon as $N\geq q^{3/2+o(1)}$. Hence, we are
left with a single critical range
\[ M=q^{1/2+o(1)},\ N^*=q^{1/2+o(1)}, \]
which is called the \emph{P\'olya--Vinogradov range} (in analogy with, say, character sums modulo $q$, where sums of length $q^{1/2+o(1)}$ are precisely the longest sums for which an application of the P\'olya--Vinogradov inequality does not shorten the sum). It is now sufficient to improve the trivial bound on \eqref{Kloosbil} in this most stubborn range.

\par
\subsubsection*{The cuspidal case} When $f$ and $g$ are both cuspidal,
the range of the variables is so short that we don't see a way to
exploit the automorphic origin of the sequence $(\lf(m))_n$ and $(\lamg(n))_n$. Instead, based on earlier work of Fouvry and Michel~\cite{AnnENS}, we
prove in Proposition \ref{bilinearpr} the following bound conditional on a square-root cancellation bound for certain complete $3$-dimensional sums of products of Kloosterman sums: 

\begin{bound}
Assume Conjecture~\ref{completesumconj}. Then, for $(a,q)=1$ and $U,V$ satisfying
\[ q^{\frac{1}{4}}\leq UV\leq q^{\frac{5}{4}}\hbox{ and } 1\leq U\leq q^{\frac{1}{4}}V \]
one has
\begin{equation}
\label{TypeIIEstimate}
\sumsum_{u\leq U,v\leq V}{\alpha_u}\beta_v\Kl(auv;q)\leq q^{o(1)}\|\alpha\|_2\|\beta\|_2(UV)^{1/2}\bigl(U^{-\frac1{2}}+q^{\frac{11}{64}}(UV)^{-\frac{3}{16}}\bigr).
\end{equation} 
\end{bound}
In the P\'olya-Vinogradov range $U\asymp V\asymp q^{1/2}$, the above bound saves a
factor $q^{1/64}$ over the trivial bound, leading to Theorem~\ref{cuspthm}. 

\subsubsection*{Non-cuspidal cases}
If $f$ or $g$ is the Eisenstein series $E$, we can exploit the
decomposition of the Hecke eigenvalues $d(n)=(1\star 1)(n)$ as a
Dirichlet convolution to obtain our unconditional results. 

First   assuming that $g=E$, the bilinear form
\eqref{Kloosbil} transforms into trilinear forms with two smooth variables of the type
\[ \sum_{m\asymp M}\sum_{n_1\asymp N_1} \sum_{ n_2\asymp
  N_2}\lambda_f(m)\Kl(\pm mn_1n_2;q)\quad \text{ with }N_1N_2= 
N^*. \]

We can then group the variables differently to form a new variable ($v$ say) whose length $V$ is {\em larger} than the P\'olya-Vinogradov range $q^{1/2}$ and apply the following bound (see \eqref{typeIIgen} of Theorem  \ref{CombinedTheorem}):
\[ B(\Kl,\bfalpha_U,\bfbeta_V )\ll q^{o(1)}(UV)^{1/2} \| \alpha \|_2 \| \beta \|_2 \bigl(U^{-1/2}+q^{1/4}V^{-1/2}\bigr). \]
This bound is non-trivial if $U\geq q^{o(1)}$ and $V\geq q^{1/2+o(1)}$.

 If such a grouping is not possible (because $N_1$ or $N_2$ is small), then we are essentially in a situation
corresponding to bilinear forms where both variables are in the P\'olya-Vinogradov range but one of them is \emph{smooth}. The main new result of this paper regarding bilinear sums of Kloosterman sums (see \eqref{typeIgen} of Theorem \ref{CombinedTheorem}) makes it possible to handle this case:

\begin{bound}
For $(a,q)=1$, and $U,V$ satisfying
\begin{equation*}1\leq U,V \leq q, \quad  UV\leq q^{3/2},\quad U\leq V^2
\end{equation*}
one has 
\begin{equation}
\label{TypeIEstimate}
  \sumsum_{u\leq U,\,v\sim V}\alpha_u\Kl(auv;q)\ll q^{o(1)} \|\alpha\|_2U^{1/2}V\bigl(q^{\frac14}U^{-\frac1{6}}V^{-\frac{5}{12}}). 
\end{equation}
\end{bound}

In the critical range $U\asymp V\asymp q^{1/2}$, the above bound saves a
factor $q^{1/24}$ over the trivial bound and this combined with arguments from \cite{MY} eventually leads to the exponent $1/68$ in the error term of Theorem~\ref{mixedthm}.

\subsubsection*{The double Eisenstein case} Finally,  in the case $f=g=E$ of Young's Theorem, we may now decompose
combinatorially both variables $m$ and $n$. Thus we reduce to
quadrilinear forms
\[ \multsum_{\substack{m_1,m_2,n_1,n_2\\ m_i\asymp M_i, n_i\asymp N_i}} \Kl(\pm
m_1m_2n_1n_2;q), \]
where
\[ M_1M_2=M\asymp  q^{1/2+o(1)},\ N_1N_2=N^*\asymp q^{1/2+o(1)}. \]
\par
We now have more possibilities for grouping variables. Especially when
two of the variables (say $m_2$ and $n_2$) are small, the grouping of $m_1,n_1$ into a single long variable $n=m_1n_1$ weighted by a divisor-like function $(1_{M_1}\star 1_{N_1})(n)$ makes it possible to use the general results of~\cite{FKM2} which provide
quite strong (unconditional) bounds for such types of sums (see \eqref{eq-fkm1} of Theorem \ref{CombinedTheorem}).

\par\medskip\par
The first step in proving \eqref{TypeIIEstimate} and \eqref{TypeIEstimate} in Section~\ref{exponentialsums} is an elaboration of Karatsuba's variant of Burgess's method along the lines of the work of Fouvry and Michel~\cite{AnnENS}. Using this,  bounds for short bilinear sums such as \eqref{Kloosbil} (strong in the critical ranges for us) can be obtained if one has upper bounds of the expected square root strength for multivariable \emph{complete} exponential sums. We prove such bounds in the situation of \eqref{TypeIEstimate} by using the Riemann Hypothesis over finite fields of Deligne~\cites{We,WeilII} and a general criterion due to Hooley~\cite{Hoo} and Katz~\cite{Sommes}.

All precise statements concerning bilinear sums of Kloosterman sums above are found in Theorem \ref{CombinedTheorem} in 
Section~\ref{sec-kloos-statements}.

\begin{remark}\label{introrem}
  (1) The combination of these arguments leads, in the special case of
  Young's Theorem, not only to stronger results, but also to a different and perhaps more
  streamlined approach. 
\par
(2) The mixed asymptotic formula of Theorem \ref{mixedthm} with \emph{some} power saving error term could be obtained   by combining the arguments of \S \ref{sec-shifted} with either Young's argument or the ones of \S \ref{exponentialsums}, but it is the combination of the three which makes it eventually possible to reach the saving $q^{1/68}$.
\par
(3) One of the known technical difficulties in the mixed case is that the variables $m$ and $n$ (and their corresponding ranges) do not play the same role. However, applying the Voronoi summation formula \emph{twice} (in different variables) allows us to essentially exchange the roles of $m$ and $n$ in critical ranges (roughly speaking, turning $B^{\pm}_{f,g}(M,N)$ via $C^{\pm}(M,N')$ to a sum of $B^{\pm}_{f,g}(M',N')$, with $M'\leq q^2/M$, $N'\leq q^2/N$); see Subsection~\ref{742}.
\par
(4) When $q$ is suitably \emph{composite}, a bilinear form in Kloosterman sums~\eqref{Kloosbil} has been estimated in \cite{BloMil} by developing a large sieve-type bound for Kloosterman sums using a variant of $q$-van der Corput method.
\par

\end{remark}

The structure of the paper is as follows. We collect some standard facts in Section~\ref{sec-reminders}.  As we have seen, the proof  depends on two crucial ingredients, the treatment of the shifted convolution sum problem and estimates of the bilinear sums of Kloosterman sums; these are the topics of Sections~\ref{sec-shifted} and \ref{exponentialsums}, respectively, while Section \ref{You} recalls briefly M.~Young's method.  Finally, Section~\ref{sec5} combines these inputs and presents the formal proofs of Theorems~\ref{improvedyoung}--\ref{cuspthm}.

\subsection*{Acknowledgments} We would like to thank Ian Petrow and Paul Nelson for many discussions on M.~Young's work. We would also thank the referees for a careful reading and useful suggestions. 
\subsection{Notation and conventions}

In the rest of this paper we will denote generically by $W$, sometimes with 
  subscripts, some smooth
complex-valued functions, compactly supported on
$[1/2,2]$, 
whose derivatives satisfy
\begin{equation}\label{Wbound}
 W^{(j)}(x)\ll_{j, \eps} q^{j\eps},
\end{equation}
for any $\eps>0$ and any $j\geq 0$, the implied constant depending on
$\eps$ and $j$ (but not on $q$).  Sometimes even the stronger bounds
$ W^{(j)}(x)\ll_j 1$ hold.  
\par
From time to time, we will use the $\eps$-convention, according to
which $\eps>0$ is an arbitrarily small positive number whose value may
change from line to line (e.g., the value of $\eps$ in \eqref{Wbound}
may be different for different functions $W$). 
\par
We denote $e(z)=e^{2\pi iz}$, and for $c\geq 1$ an integer and
$a\in\Zz$, we let $e_c(a)=e(a/c)$ be the additive character modulo
$c$. We denote by
\[ S(a, b; c) = \sum_{\substack{d\mods {c}\\(d,c)=1}} e_c(a d + b \bar{d}) \]
the usual Kloosterman sum, and we also write
\begin{equation}\label{Kloostermandefined}
\Kl(a;c)=\frac{1}{\sqrt{c}}S(a,1;c)
\end{equation}
for the normalized Kloosterman sum.
\par
We will use partitions of unity repeatedly in order to decompose a
long sum over integers into smooth localized sums (see, e.g.,
\cite{FoCrelle}*{Lemme 2}):

\begin{lemma}\label{partition} There exists a smooth non-negative function $W(x)$ supported on $[1/2,2]$ and satisfying \eqref{Wbound} such that
\[ \sum_{k\geq 0}W\Bigl(\frac{x}{2^k}\Bigr)=1 \]
for any $x\geq 1$.
\end{lemma}

\section{Arithmetic and analytic reminders}\label{sec-reminders}

We collect in this section some known preliminary facts concerning
$L$-functions and automorphic forms.  For many readers, it should be
possible to skip this section in a first reading. 

\subsection{Functional equations for Dirichlet \texorpdfstring{$L$-functions}{L-functions}}
Let $\chi $ be a non-principal character modulo  a prime $q>2$, and let $L(\chi,s)$ be its associated $L$-function. It  admits an analytic continuation to $\Cc$ and satisfies a functional equation which we now recall (see \cite{IwKo}*{Theorem 4.15} for instance): let
 \begin{equation}\label{eqparity}
   \mathfrak{a}(\chi) =\mfa=\frac{1-\chi(-1)}2= 
  \begin{cases}
  0&\text{ if } \chi (-1) =1,\\
  1&\text{ if } \chi (-1) =-1,
  \end{cases}	
 \end{equation}
and let 
\[ \Lambda (\chi, s) =q^{s/2}L_\infty(\chi,s)L(\chi, s), \quad L_\infty(\chi,s) =\pi^{-\frac{s}{2}}\Gamma \Bigl( \frac{s+\mathfrak a }{2}\Bigr) \]
be the completed $L$--function. 
 For $s\in\Cc$ one has 
\[  \Lambda (\chi, s) = \eps(\chi)  \Lambda (\overline \chi, 1-s), \]
where
\[ \eps(\chi)=i^{-\mathfrak a}\eps_\chi,\ \eps_\chi=\frac{\tau(\chi)}{\sqrt{q}}, \quad \tau (\chi)= \displaystyle\sum_{x\bmod q} \chi (x)
e (x/q).  \]
Let 
\[ L(E\otimes\chi,s)=L(\chi,s)^2,\quad L_\infty(E\otimes\chi,s)=L_\infty(\chi,s)^2,\quad\hbox{and }\Lambda(E\otimes\chi,s)=\Lambda(\chi,s)^2. \]
We deduce from the above functional equations that 
\begin{equation*}
 \Lambda(E\otimes\chi, s)   = \chi(-1)\eps^2_\chi\Lambda(E\otimes\overline \chi, 1-s).\end{equation*}  

\subsection{Cusp forms}We now describe the functional equation when $E$ is replaced by a cuspidal Hecke eigenform (holomorphic or Maa\ss)  $f$ for the group $\Gamma_{0}(1)={\rm SL}(2,\Zz)$. Let $(\lambda_{f}(n))_{n\geq 1}$ be the sequence of Hecke eigenvalues of $f$ or equivalently the coefficients of its Hecke $L$-function: 
\[ L(f, s):=\sum_{n\geq 1}\frac{\lambda_{f}(n)}{n^s}= \prod_{p}
\Bigl(1-\frac{\lambda_{f}(p)}{p^s} +\frac{1}{p^{2s}}\Bigr)^{-1}, \quad \Re s >1. \]
The numbers $\lf(n)$ satisfy the multiplicativity relations
\[ \lf(m)\lf(n)=\sum_{d|(m,n)}\lf\left(\frac{mn}{d^2}\right),\ \lf(mn)=\sum_{d|(m,n)}\mu(d)\lf\left(\frac md\right)\lf\left(\frac nd\right). \]

If $f$ is holomorphic, the Ramanujan--Petersson conjecture is known by the work of Deligne \cite{We}, and one has 
\begin{equation}\label{RPholo}
|\lf(n)|\leq d(n)	
\end{equation}
 where $d(n)$ is the divisor function. If $f$ is a Maa{\ss} form with Laplace eigenvalue $\lf(\infty)=(\frac{1}2+it)(\frac12-it)$, it follows from the work of Kim-Sarnak \cite{KiSa} that
\begin{gather}|\lf(n)|\leq d(n)n^{\theta}\hbox{ for $\theta=7/64$}\nonumber \\ \hbox{and } \label{RPbound}\\
\hbox{ either $t\in\Rr$ or $t\in i\Rr$ with $|t|\leq\theta$}.	\nonumber
\end{gather}
The Ramanujan--Petersson conjecture (that one could take $\theta=0$ in the above bounds) is at least true on average in the following sense:   for any $x\geq 1$ and any $\eps>0$, one has
\begin{equation}\label{RP4}
\sum_{n\leq x}|\lf(n)|^2\ll_{\eps,f} x^{1+\eps}.	
\end{equation}
Of course, this bound holds also for the divisor function in place of $\lambda_f$. 
 
\subsection{Functional equations for twisted \texorpdfstring{$L$-functions}{L-functions}}  For  a primitive Dirichlet character $\chi$ of prime modulus $q$, the sequence  $(\lambda_{f }(n) \chi (n))_{n\geq 1}$ is the sequence  of coefficients of the Hecke $L$-function of a cusp form $\ftchi$ relative to the group $\Gamma_{0}(q^2)$ with nebentypus $\chi^2$ (see \cite{IwKo}*{Propositions 14.19 \& 14.20}, for instance). The twisted $L$-function 
\[ L(f\otimes \chi, s):=\sum_{n\geq 1}\frac{\lambda_{f}(n)\chi (n)}{n^s}= \prod_{p}
\Bigl(1-\frac{\lambda_{f}(p)\chi (p)}{p^s} +\frac{\chi^2
(p)}{p^{2s}}\Bigr)^{-1}, \quad \Re s >1,  \]

has  an analytic continuation to  $\Cc$ and satisfies the functional equation
(see e.g.\ \cite{IwKo}*{Theorem 14.17, Proposition 14.20})
\[ \Lambda(f\otimes\chi,s) = \eps(\ftchi) \Lambda(f\otimes \overline\chi, 1-s), \]
where
\[ \Lambda(f\otimes\chi,s)=q^{s}L_\infty(f\otimes\chi,s) L(f\otimes \chi, s), \]
\[ L_\infty(f\otimes\chi,s)=\begin{cases} \Gamma_\Cc(\frac{k-1}2+s) \hbox{ if $f$ is holomorphic of weight $k$},\\
 \Gamma_\Rr(s+it+\mfa)\Gamma_\Rr(s-it+\mfa)\hbox{ if $f$ is a Maa{\ss} form with eigenvalue $(\frac{1}2+it)(\frac{1}2-it)$}
\end{cases} \]
with
\[ \Gamma_\Rr(s)=\pi^{-s/2}\Gamma(s/2),\ \Gamma_\Cc(s)=(2\pi)^{-s} \Gamma(s), \]
and the root number $\eps(\ftchi)$ is defined by
\begin{equation}\label{twistedroot}
\eps(\ftchi)= \begin{cases} \eps(f)\eps_\chi^2,&\hbox{ if $f$ is holomorphic,}\\
 	\eps(f)\chi(-1)\eps_\chi^2,&\hbox{ if $f$ is a Maa{\ss} form,}
 \end{cases}
\end{equation}
where $\eps(f)=\pm 1$ is the root number of $L(f,s)$. 
Consequently one has the following equations: 
\begin{lemma}\label{propfcteqn}  Let $f,g$ be either cuspidal Hecke eigenforms  of level $1$ or the non-holomorphic Eisenstein series $E$. Then one has, setting $\eps(E)=1$,  
\begin{equation*}
\Lambda(f\otimes\chi,s)\Lambda(g\otimes\ov\chi,s)=\eps(f,g,\chi)
\Lambda(f\otimes\ov\chi,1-s)\Lambda(g\otimes\chi,1-s),
\end{equation*}
where
\begin{align*}
  \eps(f,g,\chi)&=\eps(f)\eps(g)\text{ for $f$ and $g$ both holomorphic or both non-holomorphic,}\\
    \eps(f,g,\chi)&=\chi(-1)\eps(f)\eps(g)\text{ for $f$ holomorphic and $g$ non-holomorphic}.
\end{align*}
\end{lemma}

\begin{remark}\label{epsrem}
Observe that the root number $\eps(f,g,\chi)$ depends on $\chi$ at most through its parity $\chi(-1)$ and does not depend on $\chi$ at all if $f$ and $g$ are both holomorphic or both non-holomorphic. We will therefore denote it by $\eps(f,g,\pm 1)$ where $\pm 1=\chi(-1)$.

\end{remark}

Next, we state a standard approximate functional equation. We have (similarly as in  \cite{IwKo}*{Theorem 5.3}) the formula
\begin{multline}\label{fcteqn}
  L(f \otimes \chi, 1/2) \overline{L(g \otimes \chi, 1/2)} =
  \sum_{m,n\geqslant 1} \frac{\lambda_f(m)
    \lamg(n)}{(mn)^{1/2}}\chi(m) \bar{\chi}(n)
  V_{f,g,\pm1}\left(\frac{mn}{q^2}\right) \\+ \eps(f,g,\pm1)
  \sum_{m,n\geq 1}\frac{\lambda_f(n) \lamg(m)}{(mn)^{1/2}}\chi(m)
  \bar{\chi}(n) V_{f,g,\pm1}\left(\frac{mn}{q^2}\right)
\end{multline}
 and
\begin{equation}\label{Vfgdef}
  V_{f,g,\pm1}(x) = \frac{1}{2\pi i} \int_{(2)} \frac{L_\infty(f\otimes\chi,1/2+s)L_\infty(g\otimes\ov\chi,1/2+s)}{L_\infty(f\otimes\chi,1/2)L_\infty(g\otimes\ov\chi,1/2)} x^{-s} \frac{ds}{s}.
\end{equation}
Note that this function depends on $\chi$ at most through its parity
$\chi(-1)$, and does not depend on $\chi$ at all if $f$ and $g$ are
both holomorphic.

\subsection{Voronoi summation and Bessel functions}

The next  lemma is a version of the Voronoi formula.

\begin{lemma}\label{Voronoi}   \cite{FGKM}*{Lemma 2.2}
  Let $c$ be a positive integer and $a$ an integer coprime to $c$, and let $W$ be a smooth function compactly supported in $]0,\infty[$. Let $a_f(n)$ denote Hecke eigenvalues of a Hecke eigenform $f$ of level 1. Then
 \begin{equation*}
 \begin{split}
 & \sum_{n\geq 1} a_f(n)W(n)e\Bigl(\frac{an}{c}\Bigr)\\
  & = 
\delta_{f = E} \frac{1}{c}  \int_0^{+\infty}{(\log
    x+2\gamma-2\log c)W(x)dx}
+  \frac{1}{c} \sum_{\pm} \sum_{n\geq 1}a_f(n)\widetilde W_{\pm} \Bigl(\frac{n}{c^2}\Bigr)
  e\Bigl(\mp\frac{\overline{a}n}{c}\Bigr),\label{eq-voronoi-d}
  \end{split}
\end{equation*} 
   where $\gamma$ is Euler's constant and  the transforms $
\widetilde W_{\pm}:(0,\infty)\to\mathbb{C}$ of $W$ are defined by
\begin{equation*}
  \widetilde W_{\pm}(y)  =
   \displaystyle{ \int_0^\infty W(u)\mathcal{J}_{\pm}(4\pi \sqrt{ uy}) d
  u}
  \end{equation*}
with
\[ \mathcal{J}_+(x) = 2\pi i^k J_{k-1}(x), \quad \mathcal{J}_-(x) = 0 \]
if $f$ is holomorphic of weight $k$ and 
\[ \mathcal{J}_+(x) =  -\frac{\pi}{\cosh(\pi t)} (Y_{2it}(x) + Y_{-2it}(x)) =  \frac{\pi i}{\sinh(\pi t)} (J_{2it}(x) - J_{-2it}(x)), \]
\[ \mathcal{J}_-(x) = 4 \cosh(\pi t) K_{2it}(x) \]
 if $f$ is non-holomorphic with spectral parameter $t$ (in particular  $t=0$ if $f = E$). 
\end{lemma}
 
For the basic facts concerning the Bessel functions
$J$, $Y$ and $K$ see \cite{IWblueYellow}*{Appendix B}. In
particular, $\mathcal{J}_-$ is rapidly decaying:
\begin{equation}\label{K0}
  \mathcal{J}_-(x) \ll x^{-1/2} e^{-x}
\end{equation}
for $x \geq 1$ (and fixed $t \in\mathbb{R}$). At one point we shall need the uniform bounds
\begin{equation}\label{bessel-unif1}
  J_{it}(x) \ll e^{|t|/2} (|t| + x)^{-1/2}, \quad  t \in \Rr, x > 0
\end{equation}
and 
\begin{equation}\label{bessel-unif2}
  J_k(x) \ll \min\big(k^{-1/3}, |x^2-k^2|^{-1/4}\big), \quad k >0, x > 0. 
\end{equation}
The first bound follows from the power series expansion
\cite{GR}*{8.402} for $x < t^{1/3}$ (say) and from the uniform expansion
\cite{EMOT}*{7.13 formula (17)} otherwise.  The second bound follows
also from the power series expansion for $x < k^{1/3}$ and from
Olver's uniform expansion \cite{Ol}*{(4.24)}.
   
Integration by parts in combination with \cite{GR}*{8.472.3} shows the
formula
 \begin{multline}\label{Y0}
  \int_0^{\infty} W(y) Y_j(4 \pi \sqrt{yw + z}) dy\\ = \int_0^{\infty} \left(\frac{j }{4\pi \sqrt{yw+z}} W(y) - \frac{\sqrt{yw + z}}{2\pi w} W'(y)\right) Y_{j+1}(4 \pi \sqrt{yw + z}) dy 
\end{multline}
for $j \in\mathbb{C}$ and any smooth compactly supported function $W$.
Analogous formulae hold for $J$ and $K$ in place of $Y$.  We have the
well-known asymptotic formula \cite{GR}*{8.451.2}
\begin{equation}\label{Y}
  Y_{it}(x) = F_+(x) e^{ix} + F_-(x) e^{-ix} + O(x^{-A})
\end{equation}  
for $x \geq 1$, $t \in \mathbb{R}$ with smooth, non-oscillating functions $F_{\pm}(x)$
(depending on $t$) satisfying
\[ x^jF^{(j)}_{\pm}(x) \ll_{j, t} x^{-1/2}. \]

Finally, we consider the decay properties of the Bessel transforms $\widetilde W,\widetilde W_\pm$.

\begin{lemma}\label{besseldecay} Let $W$ be a smooth function
  compactly supported in $[1/2,2]$ and satisfying \eqref{Wbound}. In the non-holomorphic case set $\vartheta = \Re it$, otherwise set $\vartheta = 0$.   
    For
  $M\geq 1$ let $W_M(x)=W(x/M)$. For any $\varepsilon$, for any $i,j\geq 0$ and for all
  $y>0$, we have
\begin{displaymath}
\begin{split}
  y^j\widetilde{(W_M)}^{(j)}_{\pm}(y) &\ll_{i,j,\varepsilon} M(1+My)^{j/2}
 \big (1+ (My)^{-2\vartheta -\varepsilon}\big)\big(1 +   (My)^{1/2}  q^{-\varepsilon}\big)^{-i}.
\end{split}
\end{displaymath}
In particular, the functions 
$\widetilde{(W_M)}_{\pm}(y)$ decay rapidly when $y\gg q^{3\varepsilon}/M$.
\end{lemma}

\begin{proof} We differentiate $j$ times under the integral sign, followed by $i$ applications of \eqref{Y0} (or analogous formulae for $K$ and $J$) with $z=0$. Then we estimate trivially, using $\mathcal{B}_{\nu}'(x) = \frac{1}{2}(\pm \mathcal{B}_{\nu-1} - \mathcal{B}_{\nu+1})$ for $\mathcal{B} \in \{J, Y, K\}$ and the simple bounds 
\[ J_{\nu}(x) \ll_{\nu} \begin{cases} 1, & x \geq 1,\\ x^{|\Re \nu|}, & x < 1,  \nu \in \mathbb{N}_0,\end{cases} \quad Y_{\nu}(x), K_{\nu}(x) \ll_{\nu} \begin{cases} 1, & x \geq 1,\\ (1 + \log |x|)x^{-|\Re \nu|}, & x < 1.\end{cases} \qedhere \]
\end{proof}

\subsection{Kuznetsov formula and large sieve}
Next we prepare the scene for the Kuznetsov formula. We follow the notation of \cite{BHM}. 
We define the following integral transforms for a smooth function $\phi : [0, \infty) \rightarrow \Cc$ satisfying $\phi(0) = \phi'(0) = 0$, $\phi^{(j)}(x) \ll (1+x)^{-3}$ for $0 \leq j \leq 3$:
\begin{align*}
\dot{\phi}(k) &= 4i^k \int_0^{\infty} \phi(x) J_{k-1}(x) \frac{dx}{x},\\
\tilde{\phi}(t) &= 2\pi i \int_0^{\infty} \phi(x) \frac{J_{2it}(x) - J_{-2it}(x)}{\sinh(\pi t)} \frac{dx}{x},\\
 \check{\phi}(t) &= 8 \int_0^{\infty} \phi(x) \cosh(\pi t) K_{2it}(x)
 \frac{dx}{x}. 
\end{align*}

We let $\mathcal{B}_k$ be an orthonormal basis of the space of holomorphic cusp forms of level 1 and weight $k$, and we write the Fourier expansion of $f\in\mathcal{B}_k$ as
 \begin{displaymath}
  f(z) = \sum_{n \geq 1} \rho_f(n) (4\pi n)^{k/2} e(nz).
\end{displaymath}
Similarly, for Maa{\ss}  forms $f$ of level $1$ and spectral parameter $t$ we write
\begin{equation*}
  f(z) = \sum_{n \not= 0} \rho_f(n) W_{0, it}(4\pi |n|y) e(nx),
\end{equation*}
where $W_{0, it}(y) = (y/\pi)^{1/2} K_{it}(y/2)$ is a Whittaker function. We fix an orthonormal basis $\mathcal{B}$ of Hecke-Maa{\ss} eigenforms. Finally, we write the Fourier expansion of the (unique) Eisenstein series 
$E(z, s)$  of level 1  at $s = 1/2 + it$   as 
\begin{displaymath}
  E (z, 1/2 + it) =  y^{1/2+it} + \varphi(1/2 + it) y^{1/2-it} + \sum_{n \not= 0} \rho(n, t) W_{0, it}(4\pi |n|y) e(nx). 
\end{displaymath}
Then the following spectral sum formula holds (see e.g. \cite{BHM}*{Theorem 2}).

\begin{lemma}[Kuznetsov formula]\label{kuznetsov}  Let $\phi$ be as in the previous paragraph, and let $a, b > 0$   be integers. Then
\begin{multline*}
  \sum_{  c\geq 1} \frac{1}{c}S(a, b; c) \phi\left(\frac{4\pi \sqrt{ab}}{c}\right) =  \sum_{\substack{k \geq 2\\ k \text{ even}}} \sum_{f \in \mathcal{B}_k } \dot{\phi}(k) \Gamma(k) \sqrt{ab}  {\rho_f(a)} \rho_f(b)\\
  + \sum_{f \in \mathcal{B} } \tilde{\phi}(t_f) \frac{
    \sqrt{ab}}{\cosh(\pi t_f)} {\rho_f(a)} \rho_f(b) + \frac{1}{4\pi }
  \int_{-\infty}^{\infty}\tilde{\phi}(t) \frac{ \sqrt{ab}}{\cosh(\pi
    t)} {\rho(a, t)} \rho (b, t) dt
\end{multline*}
and
\begin{multline*}
  \sum_{   c\geq 1} \frac{1}{c}S(a, -b; c) \phi\left(\frac{4\pi \sqrt{ab}}{c}\right) =   \sum_{f \in \mathcal{B} } \check{\phi}(t_f) \frac{ \sqrt{ab}}{\cosh(\pi t_f)}  {\rho_f(a)} \rho_f(-b) \\
   + \frac{1}{4\pi}  \int_{-\infty}^{\infty}\check{\phi}(t) \frac{  \sqrt{ab}}{\cosh(\pi t)}  {\rho (a, t)} \rho (-b, t) dt.  
 \end{multline*}
\end{lemma}

Often an application of the
Kuznetsov formula is followed directly by an application of the large
sieve inequalities of Deshouillers-Iwaniec \cite{DI}*{Theorem 2}.

\begin{lemma}[Spectral large sieve]\label{largesieve} Let $T, M \geq
  1$, and let $(a_m)$, $M \leq m \leq 2M$, be a sequence of complex
  numbers. Then all three quantities
\begin{gather*}
  \sum_{\substack{2 \leq k \leq T\\ k \text{ even}}}\Gamma(k) \sum_{f
    \in \mathcal{B}_k }\Bigl| \sum_m a_m \sqrt{m} \rho_f(m)\Bigr|^2,
  \quad\quad \sum_{\substack{f \in \mathcal{B} \\ | t_f| \leq T}
  }\frac{1}{\cosh(\pi t_f)} \Bigl| \sum_m a_m \sqrt{m} \rho_f(\pm
  m)\Bigr|^2,
  \\
  \int_{-T}^T \frac{1}{\cosh(\pi t)} \Bigl| \sum_m a_m \sqrt{m}
  \rho(\pm m, t)\Bigr|^2 dt
\end{gather*}
are bounded by
\begin{displaymath}
  M^{\varepsilon}  (T^2 + M ) \sum_{m} |a_m|^2. 
\end{displaymath}
\end{lemma}

 Finally we quote a special case of \cite{BloMil}*{Theorem 13} which is an important variant of the preceding inequalities and responsible for making our results independent of the Ramanujan--Petersson conjecture. The main point is that we do not need to factor out the integer $s$ at the cost of $s^{\theta}$. 
 \begin{lemma}\label{avoid} Let $ s \in \bfN$, $R,  T\geq 1$, and let $\alpha( r)$,  $R \leq r \leq 2R$, be any sequence of complex numbers with $|\alpha(r ) | \leq 1$. Then
\[ \sum_{\substack{f \in \mathcal{B} \\ |t_f|  \leq T  }} \frac{1}{\cosh(\pi t_f)}      \Bigl|\sum_{\substack{R \leq r \leq 2R \\ (r, s) = 1}} \alpha(r )\sqrt{r s} \rho_f(r s) \Bigr|^2 \ll (  s T R)^{\varepsilon}   (T+ s^{1/2})(T + R) R. \]
\end{lemma}

\section{Shifted convolution sums}\label{sec-shifted}

\subsection{Statements of results}

We begin by stating the results that we use concerning the shifted
convolution problem. We will then prove the new cases that we require.
\par
For fixed modular forms $f$ and $g$ as in the introduction, for test
functions $W_1$ and $W_2$ compactly supported in $[1/2, 2]$ and
satisfying \eqref{Wbound}, and for $M, N\geq 1$, we denote
\begin{multline*}
  \ET^{\pm }_{f,g}(M,N)= \frac{1}{(MN)^{1/2}} \sum_{\substack{m \equiv \pm n
    \mods q\\ m\not= n}}  \lambda_f(m)
    \lambda_g(n) W_1\Bigl(\frac{m}{M}\Bigr)
  W_2\Bigl(\frac{n}{N}\Bigr)\\
  -\frac{1}{q(MN)^{1/2}} \sum_{(m n, q) = 1}  \lambda_f(m)
    \lambda_g(n)  W_1\Bigl(\frac{m}{M}\Bigr)
  W_2\Bigl(\frac{n}{N}\Bigr) -\delta_{f = g = E}\mathrm{MT}^{od, \pm}_{E, E}(M,N),
\end{multline*}

where $\mathrm{MT}^{od, \pm}_{E, E}(M,N)$ is the off-diagonal main term discussed by Young in~\cite{MY}*{\S 6, \S 8}.

We start with the following simple bounds which follow either from the validity of the Ramanujan--Petersson conjecture for the forms in question or the unconditional individual bound \eqref{RPbound} or the averaged bound \eqref{RP4} together with  a bound for $\mathrm{MT}^{od, \pm}_{E,E}(M,N)$ given in \cite{MY}{Lemma 6.1}. 

Define
\[ \theta_g=\begin{cases}
0,&\hbox{  if $g=E$ or is holomorphic,}\\
\theta=7/64,&\hbox{ otherwise,}	\end{cases} \]
and similarly $\theta_f$.

\begin{proposition}\label{pr-shifted-trivial}
  Let $f,g$ be either $E$ or cuspidal Hecke eigenforms of
  level $1$. Let $q$ be a prime and assume that $W_1,W_2$ satisfy
  \eqref{Wbound}. 
 
We have for $1\leq M\leq N$ the bound
\begin{equation}\label{eqtrivial}
\ET^\pm_{f,g}(M,N)\ll \left( N^{\theta_g}\frac{(MN)^{1/2}}q+\delta_{f=g=E}\Bigl(\frac{M}N\Bigr)^{1/2}\right)(qMN)^\eps.
\end{equation}

\end{proposition}
\proof By \cite{MY}*{Lemma 6.1} we have
\[ \mathrm{MT}^{od, \pm}_{E,E}(M,N)\ll (qMN)^\eps (M/N)^{1/2}. \]
Using \eqref{RP4}, the second term in the definition of $\ET^\pm_{f,g}(M,N)$ is bounded by 
$\ll q^{\eps-1} (MN)^{1/2+\eps}.$
The first  is bounded by 
$\ll q^{\eps-1} N^{\theta_g} (MN)^{1/2+\eps}$
by using  \eqref{RPbound} for $g$ and \eqref{RP4} for $f$.
\qed\\

Our main result in this section is the following theorem, which improves on \eqref{eqtrivial} in the ranges of critical importance to us.

\begin{theorem}\label{shiftthmcusp} Let $f,g$ be either $E$ or cuspidal Hecke eigenforms of
  level $1$; let $q$ be a prime and assume that $W_1,W_2$
  satisfy \eqref{Wbound}. For any $\eps>0$, there exists $\eps'>0$
  such that for $N\geq M\geq 1$  and $MN\leq
  q^{2+\eps'}$, one has
  \begin{equation}\label{shiftedfourthbound} 
  \ET^{\pm}_{f,g}(M,N)  \ll q^{\eps}\Bigl( \frac{N}{qM}\Bigr)^{1/4} \left(1+ \Bigl( \frac{N}{qM}\Bigr)^{1/4} \right)+q^{-1/2 + \theta + \varepsilon}.  
\end{equation}

\end{theorem}

\begin{remark}\label{43} 
It is a very pleasing feature that the same bound holds for cuspidal and non-cuspidal automorphic forms, even though the methods are -- at least on the surface -- rather different. 
  We note that in the case $f=g=E$ the bound \eqref{shiftedfourthbound}  improves on
  \cite{MY}*{Theorem 3.3}.

\end{remark}

The remaining part of this section is devoted to the proof of
Theorem~\ref{shiftthmcusp}.

\subsection{Preliminaries} 

We start with some general remarks. We
denote
\begin{equation*}\label{quantity}
  S^{\pm}_{f, g}(M, N) := \frac{1}{(MN)^{1/2}}\sum_{\substack{m\equiv\pm n\mods q \\ m \not=  n}}{\lf(m)\lamg(n)}W_1\Bigl(\frac{m}M\Bigr)W_2\Bigl(\frac{n}N\Bigr). 
\end{equation*}
We first observe that, by applying the Mellin inversion formula to
$W_1$ and $W_2$ together with suitable contour shifts, we have
\begin{multline}\label{B0}
  \frac{1}{q(MN)^{1/2}} \sum_{m, n} \lambda_f(m) \lambda_g(n)
  W_1\Bigl(\frac{m}M\Bigr)W_2\Bigl(\frac{n}N\Bigr) =\\
  \frac{1}{q(MN)^{1/2}}\Bigl(\underset{s=1}{\mathrm{res}}L(f,s)\widehat
  W_1(s)M^s+O_{f,A}(M^{-A})\Bigr)\Bigl(\underset{s=1}{\mathrm{res}}L(g,s)\widehat
  W_2(s)N^s+O_{g,A}(N^{-A})\Bigr)
\end{multline}
for any $A\geq 0$.  
\par
If  both $f, g$ are cuspidal, or if $f$ is cuspidal and $M\geq
q^\eps$, this term is very small. In particular, if both $f$ and $g$ are cuspidal, it is enough to obtain the stated
bound for the quantity $S^{\pm}_{f,g}(M,N)$
in place of $\ET_{f, g}^{\pm}(M, N)$. In addition, at the cost of an additional error
$O(q^{2\theta-1+\varepsilon})$, which is admissible, it suffices to estimate
\[ \frac{1}{(MN)^{1/2}}\sum_{\substack{m\equiv\pm n\mods q \\ m \not=
    n\\(mn, q) = 1}}{\lf(m)\lamg(n)}W_1\Bigl(\frac{m}M\Bigr)W_2\Bigl(\frac{n}N\Bigr). \]
\par
Then the required estimate for this last quantity is contained in
\cite{BloMil}*{(3.4), (3.11)} if $N \geq 20M$. These estimates hold for non-holomorphic forms as well; see \cite{BloMil}*{Section 11}.  In the complementary case $N \asymp M$ we have 
\[ \ET_{f, g}^{\pm}(M, N) \ll  \frac{(N+M)^{1/2 + \theta+\varepsilon}}{(NM)^{1/2}} \left(\frac{N+M}{q} + 1\right) \ll q^{-1/2 + \theta+\varepsilon} \]
(recall that $NM \leq q^{2+o(1)}$) by \cite{Bl}*{Theorem 1.3} which also holds in the non-holomorphic case. This completes the proof of 
\eqref{shiftedfourthbound} in the case $f, g$ cuspidal. 
\par
\medskip
\par
We prepare similarly for the proof of \eqref{shiftedfourthbound} in the case $f$ cuspidal, $g = E$ Eisenstein to which we devote the most work of the section.   For convenience we use the Selberg eigenvalue conjecture (known in level 1), which allows to apply Lemma \ref{besseldecay} with $\vartheta = 0$.  The following argument will feature a lot of separating variables by
integral transforms, but this is only a technical necessity and has
little to do with the core of the argument. In this context we will
frequently use Lemma \ref{partition}.

For $M\leq q^{1/4}$ the right-hand side of \eqref{shiftedfourthbound}
is larger than the right-hand side of the simple bound
\eqref{eqtrivial}. We may therefore assume that $M\geq q^{1/4}$, in
which case it suffices (by \eqref{B0} again) to estimate $S^{\pm}_{f,
  g}(M, N).$ To begin with, we make no further assumption about the size of $M, N, q$ and write $P := MNq$. For simplicity, we denote $\lambda(m)=\lf(m)$. We
open the divisor function, getting
\begin{equation}\label{extra}
  (MN)^{1/2} S^{\pm}_{f, E}(M, N) = \sum_{r \not= 0} \sum_{\substack{a, b, m \geq 1\\ m \mp ab = rq}}  \lambda(m) W_1\left(\frac{ab}{N}\right) W_2\left(\frac{m}{M}\right).
\end{equation}

We localize the variable $a$ by attaching a weight function $W_3(a/A)$ where (by symmetry) 
\begin{equation}\label{sizeA}
   A \leq N^{1/2}
\end{equation}   
    and $W_3$ is a fixed smooth weight function with support in $[1/2, 2]$. Hence it suffices to estimate
\begin{equation}\label{start}
S(M, N, q, A) =  \sum_{r \not= 0} \sum_{a} \sum_{m \equiv rq \, (\text{mod }a)} \lambda(m) W_2\left(\frac{m}{M}\right) W_3\left(\frac{a}{A}\right)W_1\left(\pm \frac{m-rq}{N}\right).
\end{equation}
\par
This expression is not symmetric in $M$ and $N$, and therefore we will
now distinguish two cases according as whether $NP^{\eps}\geq M$ or
not (the reason why it is convenient to include a
$P^{\varepsilon}$-power will be clear when we treat the second case.)

\subsection{First case}\label{first}

We first assume that $NP^{\varepsilon} \geq M$.  This condition
implies $|rq| \leq N_0 := 4N P^{\varepsilon}$.  We separate variables
by Fourier inversion:
\begin{align*}
S(M, N, q, A)  = \int_{-\infty}^{\infty} W_1^{\dagger}(x)  &\sum_{1 \leq |r| \leq N_0/q} e\left(\pm\frac{rqx}{N}\right)\\
&\sum_{a} \sum_{m \equiv rq \, (\text{mod }a)} \lambda(m) W_2\left(\frac{m}{M}\right) e\left(\mp\frac{mx}{N}\right)W_3\left(\frac{a}{A}\right) dx,
\end{align*}
where $W_1^{\dagger}$ denotes the Fourier transform. We can truncate the integral at $|x| \leq P^{2\varepsilon}$ at the cost of a negligible error. We write
\[ V(z) = V_x(z) = W_2(z) e\left( \mp z \frac{xM}{N}\right), \]
so that $V$ has compact support in $[1/2, 2]$ and satisfies $V^{(j)} \ll P^{3j\varepsilon}$, uniformly in $|x| \leq P^{2\varepsilon}$, and it remains to estimate
\begin{equation}\label{sep1}
S_x(M, N, q, A) = \sum_{1 \leq |r| \leq N_0/q} e\left(\pm \frac{rqx}{N}\right)\sum_{a} W_3\left(\frac{a}{A}\right)\sum_{m \equiv rq \, (\text{mod }a)} \lambda(m)V \left(\frac{m}{M}\right) .
\end{equation}
For later purposes, we also need to separate variables $r$ and
$q$. Let $W_4$ be smooth with support in $[0, 3]$ and constantly 1 on
$[0, 2]$, and write ${V}^{\ast}(y) = {V}^{\ast}_x(y) = W_4(y) e(\pm
yxP^{\varepsilon})$. Then by Mellin inversion we have
\begin{equation}\label{sep2}
\begin{split}
S_x(M, N, q, A) & = \sum_{1 \leq |r| \leq N_0/q}  {V}^{\ast}\left(\frac{|r|q}{NP^{\varepsilon}}\right) \sum_{a} W_3\left(\frac{a}{A}\right)\sum_{m \equiv rq \, (\text{mod }a)} \lambda(m)V \left(\frac{m}{M}\right) \\
&= \int_{(\varepsilon)}\widehat{ {V}^{\ast}}(u)  \sum_{1 \leq |r| \leq N_0/q} \left(\frac{|r|q}{NP^{\varepsilon}}\right)^{-u} \sum_{a} W_3\left(\frac{a}{A}\right)\sum_{m \equiv rq \, (\text{mod }a)} \lambda(m)V \left(\frac{m}{M}\right) \frac{du}{2\pi i}. 
\end{split}
\end{equation}
We can truncate the $u$-integral at $|\Im u| \leq P^{4\varepsilon}$,
and hence it remains to estimate
\begin{equation}\label{sep3}
\tilde{S}_{u}(M, N, q, A) = \sum_{1 \leq |r| \leq N_0/q} |r|^{-u} \sum_{a} W_3\left(\frac{a}{A}\right)\sum_{m \equiv rq \, (\text{mod }a)} \lambda(m)V \left(\frac{m}{M}\right) 
\end{equation}
uniformly in $\Re u = \varepsilon$ and $|\Im u | \leq
P^{4\varepsilon}$.  We detect the congruence with primitive additive
characters modulo $d$ for $d \mid a$. By the Voronoi summation formula
(Lemma \ref{Voronoi}), the $m$-sum equals
\[ \sum_{\pm}   \sum_{d \mid a} \frac{M}{da} \sum_{m} \lambda(m) S(rq, \pm m; d)\widetilde{V}_{\pm}\left(\frac{mM}{d^2}\right). \]
By Lemma \ref{besseldecay}, 
we see that  $\widetilde{V}_{\pm}$ is again a Schwartz class function satisfying 
\[ y^j\widetilde{V}_{\pm}^{(j)}(y)  \ll_k  P^{3j\varepsilon} \left(1 + \frac{\sqrt{y}}{P^{3\varepsilon}}\right)^{-k} \]
for any $k\geq 0$. 
 This gives
\begin{displaymath}
\begin{split}
\tilde{S}_{u}(M, N, q, A) & = \sum_{\pm} \sum_{1 \leq |r| \leq N_0/q} |r|^{-u}\sum_{a} W_3\left(\frac{a}{A}\right)\sum_{d \mid a} \frac{M}{da} \sum_{m} \lambda(m) S(rq, \pm m; d)\widetilde{V}_{\pm}\left(\frac{mM}{d^2}\right)\\
& =\sum_{\pm} \frac{M}{A} \sum_{0 \not= |r| \leq N_0/q} |r|^{-u}\sum_{b}   \sum_{d } \frac{1}{d} \sum_{m} W_5\left(\frac{db}{A}\right) \lambda(m) S(rq, \pm  m;d)\widetilde{V}_{\pm}\left(\frac{mM}{d^2}\right)
\end{split}
\end{displaymath}
where $W_5(z) = W_3(z) /z$. We localize $R \leq |r| \leq 2R$  and $M^{\ast} \leq m \leq 2M^{\ast}$ with 
\begin{equation}\label{sizeR}
1 \leq R \leq \frac{4N P^{\varepsilon}}{q}, \quad 1 \leq M^{\ast} \ll \frac{P^{4\varepsilon}A^2}{Mb^2},
\end{equation}
 up to a negligible error. Then we are left with
\[ \tilde{S}_{u}(M, N, q, A, R, M^{\ast})  = \frac{M}{A}\sum_{b \leq P}   \sum_{R \leq |r|\leq 2R} |r|^{-u} \sum_{M^{\ast} \leq m \leq 2 M^{\ast}}   \lambda(m) \sum_{d } \frac{1}{d}   S(rq,   \pm m; d) \Omega\left(\frac{4\pi \sqrt{|r|qm}}{d}\right), \]
where
\[ \Omega(z) = \Omega_{m, b, r}(z) = W_5\left(\frac{4\pi \sqrt{|r|mq}b}{zA}\right) \widetilde{V_{\pm}}\left(\frac{z^2M}{(4\pi)^2|r|q}\right). \]
The support of $W_5$ restricts the support of $\Omega$ to
\[ \frac{2\pi \sqrt{M^{\ast}R q}b}{A} \leq z \leq \frac{16\pi \sqrt{M^{\ast}R q}b}{A}. \]
Let $W_6$ be a smooth weight function that is constantly 1 on $[2\pi, 16\pi]$ and vanishes outside $[\pi, 17\pi]$. Then we have by double Mellin inversion 
\begin{displaymath}
\begin{split}
    &\Omega(z)  = W_6\left(\frac{zA}{\sqrt{M^{\ast}R q} b}\right)W_5\left(\frac{4\pi \sqrt{|r|mq}b}{zA}\right) \widetilde{V}_{\pm}\left(\frac{z^2M}{(4\pi)^2|r|q}\right)  \\
 & = W_6\left(\frac{zA}{\sqrt{M^{\ast}R q} b}\right)
  \int_{(0)} \int_{(\varepsilon)} \left(\frac{4\pi
      \sqrt{|r|mq}b}{zA}\right)^{-s}\left(\frac{z^2M}{(4\pi)^2|r|q}\right)^{-t}
  \widehat{W}_5(s) \widehat{\widetilde{V}}_{\pm}(t)\frac{dt\, ds}{(2 \pi i)^2}
  \\
 & =  \int_{(0)} \int_{(\varepsilon)}
  \left(\frac{4\pi \sqrt{|r|m}}{\sqrt{M^{\ast}R} }\right)^{-s}
  \left(\frac{MM^{\ast}Rb^2}{(4\pi
      A)^2|r|}\right)^{-t}\left(\frac{zA}{\sqrt{M^{\ast}R q}
      b}\right)^{s - 2t}  \widehat{W}_5(s) \widehat{\widetilde{V}}_{\pm}(t)
  W_6\left(\frac{zA}{\sqrt{M^{\ast}R q} b}\right) \frac{dt\, ds}{(2 \pi i)^2}.
  \end{split}
\end{displaymath}
(Assuming $\vartheta = 0$ in lemma \ref{besseldecay} allows to shift the $t$-contour to $\Re t = \varepsilon$.) 
The integrals can be truncated at $|\Im s|, |\Im t| \leq
P^{4\varepsilon}$ at the cost of a negligible error.
Writing
\[ \Theta(z) = \Theta_{s, t}(z; b) = W_6\left(\frac{zA}{\sqrt{M^{\ast}R q} b}\right)
\left(\frac{zA}{\sqrt{M^{\ast}R q} b}\right)^{s - 2t}, \]
which depends
on $b$, but not on $r$ or $m$, and satisfies $z^j\Theta^{(j)}(z) \ll_j
P^{12\varepsilon j}$, we are now left with bounding
\begin{equation*}
S_{u, s, t}(M, N, q, A, R, M^{\ast})  = \frac{M}{A}\sum_{b} \Bigl|  \sum_{R \leq |r|\leq 2R} |r|^{t-\frac{s}{2}-u}   \sum_{M^{\ast} \leq m \leq 2 M^{\ast}}   \lambda(m) m^{-\frac{s}{2}}\Sigma(rq, m) \Bigr|
\end{equation*}
where
\[ \Sigma(rq, m) = \sum_{d } \frac{1}{d}   S(rq, \pm   m; d) \Theta\left(\frac{4\pi \sqrt{|r|qm}}{d}\right) \]
and $\Re t = \Re u = \varepsilon$, $\Re s = 0$, $|\Im t|, |\Im u|,
|\Im s| \leq P^{4\varepsilon}$.  This is in a form to apply the
Kuznetsov formula (Lemma~\ref{kuznetsov}). We treat in detail the case
$r > 0$, $\pm m > 0$, the other case is essentially identical.
We get
\begin{displaymath}
\begin{split} 
\Sigma(rq, m)  = & \sum_{\substack{k \geq 2\\ k \text{ even}}} \sum_{f \in \mathcal{B}_k}\dot{\Theta}(k)\Gamma(k) \sqrt{rqm} \rho_f(rq)\rho_f( m) \\
& + \sum_{f \in \mathcal{B}}  \tilde{\Theta}(t_f)  \frac{\sqrt{rqm}}{\cosh(\pi t_f)}\rho_f(rq)\rho_f( m) + \frac{1}{4\pi}  \int_{-\infty}^{\infty}  \tilde{\Theta}(t)  \frac{\sqrt{rqm}}{\cosh(\pi t)}\rho(rq, t)\rho( m, t)dt.
\end{split}
\end{displaymath}
By \cite{BHM}*{Lemma 2.1} we have
\[ \dot{\Theta}(k) \ll_{B, \varepsilon} \frac{P^{12
     \varepsilon}}{\mathcal{T}}  \left(1 + \frac{k}{  P^{13
       \varepsilon}\mathcal{T} }\right)^{-B},\quad\quad
  \tilde{\Theta}(t) \ll_{B, \varepsilon} \frac{ P^{12
      \varepsilon}}{\mathcal{T}}  \left(1 + \frac{|t|}{  P^{13
        \varepsilon}\mathcal{T} }\right)^{-B}, \]
where 
\[ \mathcal{T} = 1+ \frac{\sqrt{M^{\ast}R q}b} {A} \]
(again this uses, for simplicity, the Selberg eigenvalue conjecture, known
in the present case of level $1$.) From now on, we use $\varepsilon$-convention. By the Cauchy--Schwarz inequality we find that the contribution of the holomorphic spectrum is at most
\begin{displaymath}
\begin{split}
 \sum_{b \leq P}  \frac{P^{\varepsilon}M}{\sqrt{M^{\ast}R q}b} & \Biggl(  \sum_{\substack{2 \leq k \ll P^{\varepsilon}  \mathcal{T} \\ k \text{ even}}}  \Gamma(k)  \sum_{f \in \mathcal{B}_k} \Bigl|
  \sum_{R \leq r\leq 2R} r^{t-\frac{s}{2}-u}   \sqrt{rq} \rho_f(rq) \Bigr|^2\Biggr)^{1/2}\\
  & 
  \Biggl(  \sum_{\substack{2 \leq k \leq P^{\varepsilon}   \mathcal{T} \\ k \text{ even}}}  \Gamma(k)  \sum_{f \in \mathcal{B}_k} \Bigl|
  \sum_{M^{\ast} \leq m \leq 2 M^{\ast}}   \lambda(m) m^{-\frac{s}{2}}   \sqrt{m} \rho_f( m) \Bigr|^2\Biggr)^{1/2}. 
 \end{split}
\end{displaymath}
Using the Ramanujan conjecture for $\sqrt{q}\rho_f(q)$ and the spectral large sieve (Lemma \ref{largesieve})
, this is (recalling \eqref{sizeA}, \eqref{sizeR}) 
\begin{equation}\label{bound1a}
\begin{split}
&\sum_{b \leq P}  \frac{P^{\varepsilon}M}{\sqrt{M^{\ast}R q}b} \left(\Bigl(\frac{R M^{\ast} q b^2}{A^2}  + R\Bigr)R \right)^{1/2} \left(\Bigl(\frac{R M^{\ast} q b^2}{A^2}  + M^{\ast}\Bigr)M^{\ast}  \right)^{1/2} \\
& \ll P^{\varepsilon} \sum_{b \leq P} \frac{M}{\sqrt{q} b} \left(\frac{N}{M} + \frac{N}{q}\right)^{1/2} \left(\frac{N}{M} + \frac{A^2}{Mb^2}\right)^{1/2} \ll P^{\varepsilon} \left(\frac{N}{q^{1/2}} + \frac{N\sqrt{M}}{q} \right).
\end{split}
\end{equation}
The same argument works for the Eisenstein spectrum. For the Maa{\ss}
spectrum, we need to argue differently in order to avoid the Ramanujan
conjecture. Here we use Lemma \ref{avoid} with $s=q$ (estimating trivially the terms with $(r, q) > 1$ that only occur in the case $R \geq q$) to conclude that the total
contribution of the Maa{\ss} spectrum is
\begin{equation}\label{bound2a}
\begin{split}
   & \sum_{b \leq P}  \frac{P^{\varepsilon}M}{\sqrt{M^{\ast}R q}b}\big[((\mathcal{T}^2 + R^2) R^2)^{1/4}  (\mathcal{T}^2 + q)^{1/4}  + (\mathcal{T}^2 R^{4\theta})^{1/2}\big]((\mathcal{T}^2 +M^{\ast})M^{\ast})^{1/2} \\
    & \ll P^{\varepsilon} \sum_{b \leq P}  \frac{M}{\sqrt{q}b} \left(\frac{N}{M} + \frac{N^2}{q^2}\right)^{1/4}  \left(\frac{N}{M} + q \right)^{1/4}  \left(\frac{N}{M} + \frac{N}{Mb^2}\right)^{1/2} \\
    & \ll P^{\varepsilon} 
    \left(\frac{N}{q^{1/2}} +\frac{ M^{1/4}N^{3/4}}{ q^{1/4}}\right)\left(1 + \frac{(MN)^{1/4}}{q^{1/2}}\right). 
\end{split}
\end{equation}
Note that \eqref{bound2a} is larger than \eqref{bound1a} when $M \ll N
P^{\varepsilon}$. This completes the analysis of the contribution of
$\Sigma(rq, m)$.

\subsection{Second case}\label{second}

We now assume that $NP^{\varepsilon} \leq M$.  We return to
\eqref{start} and begin with some preliminary transformations. We
write
\[ V(z) = V_{rq}(z) := W_2\left(\frac{Nz + rq}{M}\right) = \int_{-\infty}^{\infty} W_2^{\dagger}(x) e\left(\frac{Nz + rq}{M}x\right) dx. \]
The integral can be truncated at $|x| \leq P^{\varepsilon}$ at the
cost of a negligible error. Since $W_2(m/M) = V((m-rq)/N)$, putting
$W_4(z) = W_1(z) e(\pm z Nx/M)$, we get
\begin{align*}
  W_2\left(\frac{m}{M}\right) W_1\left(\pm \frac{m-rq}{N}\right) &=
  V\left(\frac{m-rq}{N}\right)W_1\left(\pm \frac{m-rq}{N}\right) \\
  &= \int_{-\infty}^{\infty}W_2^{\dagger}(x) e\left(
    \frac{rqx}{M}\right) W_4\left(\pm \frac{m-rq}{N}\right) dx.
\end{align*}

Note that $W_4$ has support in $[1/2, 2]$ and satisfies $W_4^{(j)}
\ll_{j} 1$ uniformly in $|x| \leq P^{\varepsilon}$. Hence we are
left with
\[ S_x(M, N, q, A) =  \sum_{r \asymp M/q} e\left( \frac{rqx}{M}\right) \sum_{a} \sum_{m \equiv rq \, (\text{mod }a)} \lambda(m)  W_3\left(\frac{a}{A}\right)W_4\left(\pm \frac{m-rq}{N}\right), \]
where $r \asymp M/q$ is short for $r \in [c_1M/q, c_2M/q]$ for suitable constants $c_1, c_2$. 
As in \eqref{sep1} -- \eqref{sep3} we may separate the variables $r$ and $q$, and need to bound
\[ \tilde{S}_{ u}(M, N, q, A) =  \sum_{r \asymp M/q} r^{-u}  \sum_{a} \sum_{m \equiv rq \, (\text{mod }a)} \lambda(m)  W_3\left(\frac{a}{A}\right)W_4\left(\pm \frac{m-rq}{N}\right) \]
with $\Re u = \varepsilon$, $|\Im u | \leq P^{\varepsilon}$. Again we detect the congruence with primitive additive characters modulo $d$ for $d \mid a$ and  apply Voronoi summation (Lemma \ref{Voronoi})  to the $m$-sum getting
\begin{equation}\label{new}
\begin{split}
\tilde{S}_{ u}&(M, N, q, A) =   \sum_{r \asymp M/q} r^{-u}  \sum_{a}W_3\left(\frac{a}{A}\right) \sum_{d \mid a} \frac{N}{da}\sum_{\epsilon \in \{\pm \}}  \sum_{m} \lambda(m) S(rq, \epsilon m; d) W_4^{\epsilon}\left( \frac{ mrq}{d^2}, \pm \frac{ mN}{d^2}\right) 
\end{split}
\end{equation}
where
\begin{displaymath}
W_4^{\pm}(z, w) =   \int_0^{\infty} W_4(y) \mathcal{J}_{\pm}(4\pi \sqrt{z + wy})dy, \quad 4|w| \leq z. 
\end{displaymath}

Note that by our current size assumption $NP^{\varepsilon} \leq M$,
the first argument in $W^{\epsilon}_4(z, w)$ is substantially larger than
the second. We follow the argument of \cite{BloMil}*{Lemma 17 \& Remark after Corollary 
18}.

As
\[ \frac{mrq}{d^2} - 2\frac{mN}{d^2} \gg \frac{M - O(N)}{A^2} \gg \frac{M}{N} \gg P^{\varepsilon}, \]
the case $\epsilon = -1$
contributes negligibly due to the rapid decay of the
Bessel-$K$-function (this is another reason why we separate
  the two cases in the somewhat artificial way $NP^{\varepsilon} \geq
  M$ and $NP^{\varepsilon} \leq M$), cf.\ \eqref{K0}.  Hence it
suffices to consider only the case $\epsilon = 1$.  For later purposes
(see \eqref{doublemellin} below) it is convenient to insert into
\eqref{new} a smooth, redundant weight function $W_0(mrq/d^2, \pm
mN/d^2)$ such that $W_0(z, w) = 0$ for $z \leq 1$ or $3|w| \geq z$,
and $W_0(z, w) = 1$ for $z \geq 2$ and $4|w| \leq z$. We write $W_5(z,
w) = W_0(z, w)W_4^+(z, w)$, so that
 \begin{displaymath}
\begin{split}
\tilde{S}_{ u}&(M, N, q, A) =   \sum_{r \asymp M/q} r^{-u}  \sum_{a}W_3\left(\frac{a}{A}\right) \sum_{d \mid a} \frac{N}{da}   \sum_{m} \lambda(m) S(rq,   m; d) W_5\left( \frac{ mrq}{d^2}, \pm \frac{ mN}{d^2}\right)
\end{split}
\end{displaymath}
up to a negligible error coming from $\epsilon = -1$. 
An integral transform similar to $W_5(z, w)$ was analyzed in \cite{BloMil}*{Lemma 17}. Our general assumption in the forthcoming analysis is 
\[ z \asymp z + wy \gg P^{\varepsilon}. \]

Repeated application of the formula \eqref{Y0} yields the preliminary bound
\[ W_5(z, w) \ll_k \left(\frac{\sqrt{z}}{w}  \right)^k \]
for any $k\geq 0$. In particular, up to a negligible error of $O(P^{-k})$, we can assume that 
\begin{equation}\label{restrict}
  \sqrt{z} \geq wP^{-\varepsilon}.
 \end{equation} 
   In this range we use the asymptotic formula \eqref{Y}
 , so that   
 \[ W_5(z, w) = W_+(z, w) e(2 \sqrt{z}) + W_-(z, w) e(-2 \sqrt{z})  + {\rm O}(P^{-k}), \]
where
\begin{displaymath}
   z^i |w|^j \frac{\partial^i}{\partial z^i}    \frac{\partial^j}{\partial w^j} W_{\pm}(z, w)     \ll P^{\varepsilon(i+j)}  z^{-1/4}.
 \end{displaymath}
It is now easy to see (cf.\   \cite{BloMil}*{Corollary 18})   that 
its double Mellin transform 
\begin{equation}\label{doublemellin}
\widehat{W}_{\pm, \pm }(s, t) = \int_0^{\infty} \int_0^{\infty} W_{\pm}(z, \pm w) z^{s-1} w^{t-1} dz\, dw
\end{equation}
is rapidly decaying on vertical lines (i.e. is $\ll_{k,\ell,\eps}
P^{\varepsilon} |s|^{-k} |t|^{-\ell}$ for $|s|$, $|t|\geq 1$) and
absolutely convergent in $\Re t > 0$, $\Re s + \Re t/2 < 1/4$.

We can restrict $m$ to a dyadic range $M^{\ast} \leq m \leq 2M^{\ast}$, and \eqref{restrict} implies
\[ M^{\ast} \ll \frac{P^{2\varepsilon} M A^2}{(b N)^2}. \]
 This leaves us with bounding
 \begin{displaymath}
 \begin{split}
&\tilde{S}_{ u}(M, N, q, A, M^{\ast})\\
& =  \sum_{r \asymp M/q} r^{-u}  \sum_{b \leq P} \frac{1}{b}\sum_{d } W_3\left(\frac{bd}{A}\right)  \frac{N}{d^2} \sum_{M^{\ast} \leq m \leq 2 M^{\ast}} \lambda(m) S(rq, m; d) e\left(\pm \frac{2\sqrt{mrq}}{d}\right) W_{\pm} \left(\frac{ mrq}{d^2}, \frac{ mN}{d^2}\right)\\
 & =\sum_{b \leq P} \frac{N}{b}  \sum_{r \asymp M/q} r^{-u} \sum_{M^{\ast} \leq m \leq 2 M^{\ast}} \lambda(m)   \sum_{d }   \frac{1}{d} S(rq, m; d)   \Omega\left(\frac{4\pi \sqrt{mrq}}{d}\right), 
 \end{split}
 \end{displaymath}
 where
 \[ \Omega(z) = W_3\left(\frac{4\pi b \sqrt{mrq}}{Az}\right) \frac{z}{4\pi \sqrt{mrq}} W_{\pm}\left(\frac{z^2}{(4\pi)^2}, \frac{z^2}{(4\pi)^2} \frac{N}{rq} \right) \exp(\pm iz) \]
has support contained in
\[ z \asymp  Z := \frac{ b \sqrt{M^{\ast} M}}{A} \gg 1. \]
Once again we add a redundant weight function $W_6(z/Z)$ of compact support (to remember the original size condition) that is constantly 1 on a sufficiently large (fixed) interval, and we separate variables  by Mellin inversion: 
\begin{displaymath}
\begin{split}
\Omega(z) & = W_6\left(\frac{z}{Z}\right)  \exp(\pm iz) \\
&\times \int_{(0)} \int_{(\varepsilon)} \int_{(1/4 - \varepsilon)}  \widehat{W}_3(v)\widehat{W}_{\pm}(s, t)  \left(\frac{4\pi b \sqrt{mrq}}{Az}\right)^{-v} \frac{z}{4\pi \sqrt{mrq}}   \left(\frac{z}{4\pi}\right)^{-2s-2t} \left(\frac{N}{rq}\right)^{-t}  \frac{ds}{2\pi i} \frac{dt}{2\pi i} \frac{dv}{2\pi i}\\
& = \int_{(0)} \int_{(\varepsilon)} \int_{(1/4 - \varepsilon)} \frac{\widehat{W}_3(v)\widehat{W}_{\pm}(s, t) }{(4\pi)^{1+v-2s-2t}} W_6\left(\frac{z}{Z}\right)  \exp(\pm iz)  \left(\frac{b}{A}\right)^{-v}\\
& \quad\quad \times  (M^{\ast})^{-\frac{v+1}{2}} M^{t-\frac{v+1}{2}} Z^{1-2s-2t} N^{-t}  \left(\frac{m}{M^{\ast}}\right)^{-\frac{v+1}{2}}  \left(\frac{rq}{M}\right)^{t-\frac{v+1}{2}} \left(\frac{z}{Z}\right)^{1-2s-2t}  \frac{ds}{2\pi i} \frac{dt}{2\pi i} \frac{dv}{2\pi i}.
\end{split}
\end{displaymath}
We can truncate the integrals at $|\Im s|, |\Im t|, |\Im v| \leq P^{2\varepsilon}$ at the cost a negligible error. Hence we need to bound
\begin{displaymath}
\begin{split}
S_{ u, s, t, v}(M, N, q, A, M^{\ast}) = & \sum_{b \leq P }   \frac{NZ^{1/2}}{b\sqrt{M^{\ast}M}}\Bigl| \sum_{r \asymp M/q} r^{-\alpha} \sum_{M^{\ast} \leq m \leq 2 M^{\ast}} \lambda(m)  \left(\frac{m}{M^{\ast}}\right)^{-\frac{v+1}{2}} \\
&\times  \sum_{d }   \frac{1}{d} S(rq, m; d)   \Theta\left(\frac{4\pi \sqrt{mrq}}{d}\right)\Bigr|, 
\end{split}
\end{displaymath}
where $\alpha = \frac{v+1}{2}-t + u$ and 
\[ \Theta(z) = \Theta_{s, t}(z) = W_6\left(\frac{z}{Z}\right)  \exp(\pm iz)\left(\frac{z}{Z}\right)^{1-s-t}. \]
We apply the Kuznetsov formula (Lemma \ref{kuznetsov}) to the $d$-sum. By \cite{BloMil}*{Lemma 16}, the spectral sum can be truncated (with a negligible error) at spectral parameter $P^{3\varepsilon}Z^{1/2}$, and we obtain
\[ S_{u, s, t, v}(M, N, q, A, M^{\ast})  = \sum_{b \leq P}\frac{NZ^{1/2}}{b\sqrt{M^{\ast}M}} \Bigl| \sum_{r \asymp M/q} r^{-\alpha} \sum_{M^{\ast} \leq m \leq 2 M^{\ast}} \lambda(m)  \left(\frac{m}{M^{\ast}}\right)^{-\frac{v+1}{2}} \Bigl( \mathcal{H} + \mathcal{M} + \mathcal{E}\Bigr)\Bigr| \]
(up to a negligible error), where
\begin{displaymath}
\begin{split}
\mathcal{H} &=   \sum_{\substack{2 \leq k \ll P^{3\varepsilon} Z^{1/2} \\ k \text{ even}}} \sum_{f \in \mathcal{B}_k} 4i^k\Gamma(k) \int_0^{\infty} J_{k-1}(z) \Theta(z) \frac{dz}{z} \sqrt{mrq} \rho_f(m) \rho_f(rq),\\
\mathcal{M} & =   \sum_{\substack{f \in \mathcal{B}\\ t_f \ll P^{3\varepsilon} Z^{1/2}}} 2\pi i  \int_0^{\infty} \frac{J_{2 i t_f}(z) - J_{-2 it_f}(z)}{\sinh(\pi t_f)} \Theta(z) \frac{dz}{z} \frac{\sqrt{mrq} \rho_f(m) \rho_f(rq)}{\cosh(\pi t_f)},\\
\mathcal{E} & =   \int_{| t|  \ll P^{3\varepsilon} Z^{1/2} } \frac{i}{2}  \int_0^{\infty} \frac{J_{2 i t}(z) - J_{-2 it}(z)}{\sinh(\pi t)} \Theta(z) \frac{dz}{z} \frac{\sqrt{mrq} \rho(m, t) \rho(rq, t)}{\cosh(\pi t)} dt 
\end{split}
\end{displaymath}
denote the respective contributions of the holomorphic cusp forms, non-holomorphic (Maa{\ss}) cusp forms and of the Eisenstein series.
It follows from \eqref{bessel-unif1} and \eqref{bessel-unif2} that 
\[ J_{k-1}(z), \frac{J_{2 i t}(z) - J_{-2 it}(z)}{\sinh(\pi t)}  \ll   P^{\varepsilon} Z^{-1/2}, \quad z \asymp Z, \quad t, k \ll P^{3\varepsilon} Z^{1/2}. \]
(Indeed, if $k \asymp Z$, then $Z \ll P^{6\varepsilon}$ and $k^{-1/3} \asymp Z^{1/6} Z^{-1/2} \ll P^{\varepsilon} Z^{-1/2}$.) 
We estimate the $z$-integral trivially. From now on we use $\varepsilon$-convention. The Maa{\ss} contribution is at most
\begin{displaymath}
\begin{split}
 \sum_{b \leq P}  \frac{P^{\varepsilon} N}{b\sqrt{M^{\ast}M}}  \sum_{\substack{f \in \mathcal{B}\\ t_f \ll P^{3\varepsilon} Z^{1/2}}}  \Bigl| \sum_{r \asymp M/q} r^{-\alpha} \sum_{M^{\ast} \leq m \leq 2 M^{\ast}} \lambda(m)  \left(\frac{m}{M^{\ast}}\right)^{-\frac{v+1}{2}} \frac{\sqrt{mrq} \rho_f(m) \rho_f(rq)}{\cosh(\pi t_f)} \Bigr|, 
\end{split}
\end{displaymath}
and similar expressions hold for the holomorphic and Eisenstein contribution. By the Cauchy--Schwarz inequality this is at most
\begin{displaymath}
\begin{split}
\sum_{b}  \frac{P^{\varepsilon} N }{\sqrt{bM^{\ast}M}}   & \Biggl(\sum_{\substack{ f \in \mathcal{B}\\ t_f \ll P^{\varepsilon} Z^{1/2}}}  \Bigl| \sum_{r \asymp M/q} r^{-\alpha}  \frac{\sqrt{rq}  \rho_f(rq)}{\cosh(\pi t_f)} \Bigr|^2\Biggr)^{1/2} \\
& \times  \Biggl(
    \sum_{\substack{f \in \mathcal{B}\\ t_f \ll P^{\varepsilon} Z^{1/2}}}  \  \sum_{M^{\ast} \leq m \leq 2 M^{\ast}} \lambda(m)  \left(\frac{m}{M^{\ast}}\right)^{-\frac{v+1}{2}} \frac{\sqrt{m } \rho_f(m)  }{\cosh(\pi t_f)} \Bigr|^2\Biggr)^{1/2}. 
 \end{split}
\end{displaymath}
Using Lemmas \ref{largesieve} -- \ref{avoid} as in the previous case, this is
\begin{displaymath}
\begin{split}
 & \ll   \sum_{b \leq P} \frac{P^{\varepsilon} N }{b\sqrt{M^{\ast}M}} \left[(Z + q)^{1/4} \left(\Bigl(Z+  \frac{M^2}{q^2} \Bigr) \frac{M^2}{q^2}\right)^{1/4} + \left(\Bigl(Z + \frac{M}{q}\Bigr) \frac{M}{q}  \right)^{1/2}\right]((Z + M^{\ast})M^{\ast})^{1/2} \\
&\ll    \sum_{b \leq P} \frac{P^{\varepsilon} N }{b\sqrt{M}} \left[\left(\frac{M}{N} + q\right)^{1/4} \left(\Bigl(\frac{M}{N}+  \frac{M^2}{q^2} \Bigr) \frac{M^2}{q^2}\right)^{1/4} + \left(\Bigl(\frac{M}{N} + \frac{M}{q}\Bigr) \frac{M}{q}  \right)^{1/2}\right]\left(\frac{M}{N} + \frac{MA^2}{(bN)^2}\right)^{1/2}. 
\end{split}
\end{displaymath}
By \eqref{sizeA} this
\begin{equation}\label{final}
\begin{split}
& \ll  P^{\varepsilon} N^{1/2}  \left[ \frac{M}{N^{1/2} q^{1/2}} + \frac{M^{5/4}}{qN^{1/4}} + \frac{M^{3/4}}{N^{1/4} q^{1/4}} + \frac{M}{q^{3/4}} \right]\\
& \ll P^{\varepsilon}  \left(\frac{M}{q^{1/2}} + \frac{M^{3/4}N^{1/4}}{ q^{1/4}}\right)\left(1 + \frac{(MN)^{1/4}}{q^{1/2}}\right).
\end{split}
\end{equation}

Combining \eqref{bound2a} and \eqref{final}, and recalling the extra factor $(MN)^{1/2}$ in \eqref{extra}, we complete the proof of \eqref{shiftedfourthbound} in the case $f$ cuspidal, $g = E$ Eisenstein. 

\subsection{The case \texorpdfstring{$f=g=E$}{f=g=E}}\label{sec-MYimprove}

Here we merely indicate the points where the proof of \cite{MY}*{Thm. 3.3}
needs some modifications. We will freely borrow the notations of that
paper, where the error term $ \ET_{E,E}(M, N)$ is denoted by
$E_{M,N}$.
\par
In \cite{MY}*{\S 9}, this term is further decomposed as a sum of two
terms $E_{M,N}=E_{+}+E_{-}$ and each of these two terms is decomposed
into cuspidal holomorphic, cuspidal non-holomorphic and Eisenstein
contributions denoted by $E_{h\pm}+E_{m\pm}+E_{c\pm}$ in that paper.
The term $E_{m-}$ is the most complicated one, and it is here that we
insert some modifications. This term decomposes further as a sum of
terms denoted by $E_K$ where $K$ is a parameter around which the
Laplace eigenvalues of the Maa{\ss} forms are localized.

In \cite{MY}*{(9.12)}, we apply H\"older's inequality with exponents
$(1/4, 1/4, 1/2)$, and obtain
\begin{multline*}
  \Bigl| \sum_{K \leq \kappa_j < 2K} \frac{|\rho_j(1)|^2}{\cosh(\pi
    \kappa_j)} \lambda_j(q) L_j(1/2 + s_1) L_j(1/2 + s_2)^2\Bigr|
  \leq  \Bigl| \sum_{K \leq \kappa_j < 2K} \frac{ |\rho_j(1)|^2}{\cosh(\pi \kappa_j)} |\lambda_j(q)|^4 \Bigr|^{1/4} \\
  \cdot \Bigl| \sum_{K \leq \kappa_j < 2K}
  \frac{|\rho_j(1)|^2}{\cosh(\pi \kappa_j)} |L_j(1/2 + s_1)|^4
  \Bigr|^{1/4}\cdot \Bigl| \sum_{K \leq \kappa_j < 2K}
  \frac{|\rho_j(1)|^2}{\cosh(\pi \kappa_j)} |L_j(1/2 + s_1)|^4
  \Bigr|^{1/2}.
\end{multline*}
\par
Exactly as in \cite{MY}*{(9.13)}, we bound the last two factors by
$K^{3/2 + \varepsilon}$. For the first factor we write
\[ |\lambda_j(q)|^4 \leq 2(|\lambda_j(1)| + |\lambda_j(q^2) |)^2 \]
and use \cite{Mot}*{Lemma 2.4} to estimate this factor by
$(qK)^{\varepsilon} (K^2+q)^{1/4}$.
 In this way, \cite{MY}*{(9.15)} becomes
\begin{equation}\label{holder}
E_{m\pm} \ll q^{-1/2+\varepsilon} \left(\frac{N}{M}\right)^{1/2} + q^{-1/4+\varepsilon} \left(\frac{N}{M}\right)^{1/4} .
\end{equation}
Then  \cite{MY}*{Proposition 9.3} and the last bound of \cite{MY}*{Section 9.6}  give two additional error terms
\[ q^{-1/2   + \varepsilon} \left(\frac{N}{M}\right)^{1/4}
 + q^{-1/2 + \theta + \varepsilon} \left(\frac{M}{N}\right)^{1/2}, \]
both of which are dominated by \eqref{holder}, at least for $\theta < 1/4$ and under the general assumption  $N \geq M$.  Hence we get an improved version of  \cite{MY}*{Theorem 3.4}:
\begin{equation}\label{young-improved}
 \ET_{E,E}(M, N) \ll q^{-1/2+\varepsilon} \left(\frac{N}{M}\right)^{1/2} + q^{-1/4+\varepsilon} \left(\frac{N}{M}\right)^{1/4}
\end{equation}
under the assumption $MN \ll q^{2+\varepsilon}$ and $N \geq M$. \qed

\begin{remark}\label{rm-young-improved}
  Inserting \eqref{young-improved} into the subsequent analysis of the
  piecewise linear function at the end of Section 3 in \cite{MY}, we
  obtain the exponent $-1/82$ in the error term of \cite{MY}*{Theorem 1} (with no $\theta$-dependence), the maximum being taken at $a
  = 21/41$ and $b = 60/41$.
\end{remark}

\section{Young's method}\label{You} 

In this section we prove the following small variation of \cite{MY}*{Lemma 4.1, 4.2}. 

\begin{proposition}\label{prop-Young}
Let $\lambda(m)$ denote either the (normalized) Fourier coefficients of a cuspidal Hecke eigenform $f$ of level $1$, or the divisor function.  Let $N, N_1, N_2, M \geq 1$ be parameters with $N_1N_2 = N$, $N_1 \leq N_2$,  and let  $q \in \mathbb{N}$. Let $W_1$, $W_2$, $W_3$ be  smooth compactly supported weight functions satisfying \eqref{Wbound}. Then
\begin{equation*} 
\frac{1}{\sqrt{MN}} \sum_{n_1n_2\equiv  \pm m\, (\text{{\rm mod }} q)} \lambda(m)     W_1(n_1/N_1)  W_2(n_2/N_2)  W_3(m/M)
\end{equation*}
is bounded by  the following two  quantities:
\[ (MNq)^{\varepsilon} \cdot  \begin{cases} \frac{\sqrt{MN}}{q^{2-\theta}} + \min\left(\frac{(Mq)^{1/2}}{N^{1/2}} + \frac{N_1M^{1/2} }{qN^{1/2}} , \frac{ q^{1/4}}{N_1^{1/2}}+ \frac{(MN_1)^{1/2}}{N^{1/2}}+ \frac{N_1M^{1/2}}{q N^{1/2}}, \frac{M^{1/2}N_1}{N^{1/2}}\right) ,\\
\frac{\sqrt{MN}}{q^{2-\theta}}+  \min\left(\frac{N_1^2}{(MN)^{1/2}}, \frac{ N^{1/6}N_1q^{1/2}}{N_2 M^{2/3} }\right)+ \frac{M^{1/2}}{N^{1/2}} + \frac{M^{1/2}N_1}{q N^{1/2}} + \frac{M^{3/2}}{N_2N^{1/2}} .
 \end{cases} \]
\end{proposition}

\begin{remark} As will become clear from the proof, the starting point is to apply Poisson summation in $n_2$. This is a very different strategy compared to the outline in Section \ref{PrincipleSection}, which dualizes the variables $n_1$, $n_2$ simultaneously in the form of Voronoi summation. 
\end{remark}

\begin{proof} We follow closely the argument in  \cite{MY}*{Lemma 4.1, 4.2} and keep track of the following two  differences. We drop the assumption $MN \leq q^{2+\varepsilon}$ and we allow that $\lambda$ can be the divisor function or the sequence of Hecke eigenvalues (and we make sure to use only  bounds   of the type \eqref{RP4} and \eqref{RPbound}).

The first bound is the analogue of \cite{MY}*{Lemma 4.1}. We can exclude the terms $n_1n_2 \equiv 0 $ (mod $q$) at the cost of an error 
$  (MNq)^{\varepsilon}  (MN)^{1/2} q^{-2+\theta}.$  We apply Poisson summation to the $n_2$-sum. The central term contributes   an error of $O(\sqrt{MN}/q^2)$, and we bound the quantity $R$ in \cite{MY}*{(4.6)} by 
\[ R \ll (MNq)^{\varepsilon} \frac{N_2}{q \sqrt{MN}} S\left(N_1, Mq/N_2, q\right), \]
where
\[ S(K, L, q) = \sum_{\substack{ l \leq L\\(l, q) = 1}} |(\lambda \ast 1)(l)|\cdot  \Bigl| \sum_{(k, q) = 1} e\left(\frac{l\bar{k}}{q}\right) W\left(\frac{k}{K}\right)\Bigr| \]
is analogous to  \cite{MY}*{(4.7)}. The proof of \cite{MY}*{Proposition 4.3} provides bounds for $S(K, L, q)$ in the situation where $\lambda$ is the divisor function and under the additional assumption $L, K \ll q^{1+\varepsilon}$. In order to also include Fourier coefficients, we notice that the proof of \cite{MY}*{Proposition 4.3} uses only $\infty$-norms or $2$-norms for the $k$-sum, so a Rankin--Selberg-type bound for $\lambda \ast 1$ suffices. Without  the condition $L, K \ll q^{1+\varepsilon}$ we obtain
\[ S(K, L, q) \ll (KLq)^{\varepsilon} \min\left(Lq^{1/2} + KL/q, (Lq^{3/2} + L^2K + L^2K^2/q)^{1/2}, LK\right), \]
where the first bound is the analogue of \cite{MY}*{(4.11)}, the second bound is the analogue of the last display in \cite{MY}*{Section 4.2} and the last bound is the trivial bound. In this way we arrive at the first bound   of our proposition.  

The second bound in Proposition \ref{prop-Young} is the analogue of \cite{MY}*{Lemma 4.2}, and again we only indicate the changes in Young's proof.  The error term in the second display of \cite{MY}*{Section 4.3}  is (recall Young's notation $H = q/N_2$) 
\[ (MNq)^{\varepsilon}\left( \frac{MN_1}{q \sqrt{MN}} + \frac{M^2}{N_2\sqrt{MN}}\right). \]
If $\lambda = d$ is the divisor function, then the pole in \cite{MY}*{(4.13)} contributes
\[ \ll (MNq)^{\varepsilon} \frac{M}{\sqrt{M N }}. \]
In either case, after shifting the contours, we apply Voronoi summation to the term $U(h, m, n_1)$ in the last line on \cite{MY}*{p.\ 22} and arrive at a quantity analogous to \cite{MY}*{(4.15)}. 
Finally, the pointwise bound on the quantity $V(h,k)$ defined under \cite{MY}*{(4.17)} and proved above \cite{MY}*{(4.18)} allows us to reach the analogue of \cite{MY}*{(4.16)--(4.17)}, which yields the second bound of our proposition (and we notice that the assumption $N \gg q^{1+\varepsilon}$ in \cite{MY}*{p.\ 24, line 8} can be assumed in our case, too, since otherwise the term $(M/N)^{1/2} $ is worse than the trivial bound). \end{proof}

For later purposes, we will also need the following immediate corollary, which we state here for easy reference.

\begin{corollary}\label{cor-Young}
Let $\lambda(m)$ denote either the (normalized) Fourier coefficients of a   cuspidal Hecke eigenform $f$ of level $1$, or the divisor function.  Let
\[ 1 \leq N' \leq N^*, \quad N_1N_2 = N', \quad N_2 \geq N_1 \geq 1, \quad 1 \leq M' \leq M^* \]
be parameters and let  $q \in \mathbb{N}$. Let $W_1$, $W_2$, $W_3$ be  smooth compactly supported weight functions satisfying \eqref{Wbound}. Then
\begin{equation*} 
\frac{1}{\sqrt{M^{\ast} N^{\ast}}} \sum_{n_1n_2\equiv  \pm m\, (\text{{\rm mod }} q)} \lambda(m)     W_1(n_1/N_1)  W_2(n_2/N_2)  W_3(m/M')
\end{equation*}
is bounded by  the following two  quantities:
\[ (M^*N^*q)^{\varepsilon} \cdot  \begin{cases} \frac{\sqrt{M'N'}}{q^{2-\theta}} + \min\left(\frac{M'q^{1/2}}{\sqrt{M^*N^*}} + \frac{N_1M'}{q\sqrt{M^{\ast}N^{\ast}}} , \frac{\sqrt{M' N_2}q^{1/4}}{\sqrt{M^*N^*}}+ \frac{M'N_1^{1/2}}{\sqrt{M^*N^*}}+ \frac{N_1M'}{q\sqrt{M^{\ast}N^{\ast}}}, \frac{M'N_1}{\sqrt{M^*N^*}}\right) ,\\
\frac{\sqrt{M'N'}}{q^{2-\theta}}+  \min\left(\frac{N_1^2}{\sqrt{M^*N^*}}, \frac{(M')^{-1/6} (N')^{2/3}N_1q^{1/2}}{N_2 \sqrt{M^*N^*}} \right)+ \frac{M'}{\sqrt{M^{\ast}N^{\ast}}} + \frac{M'N_1}{q \sqrt{M^{\ast}N^{\ast}}} + \frac{(M')^2}{N_2\sqrt{M^*N^*}} .
 \end{cases} \]
\end{corollary}

\section{Bilinear forms in Kloosterman sums}\label{exponentialsums}

\subsection{Statements of results}\label{sec-kloos-statements}

We begin by stating the precise results we obtain concerning the
bilinear forms of type~(\ref{Kloosbil}), i.e.
\[ B(\Kl,\bfalpha,\bfbeta)=\sum_m\sum_n \alpha_m\beta_n \Kl(amn;q), \]
where the normalized Kloosterman sum $\Kl(a;q)$ is defined in \eqref{Kloostermandefined}.
\par
We recall from Section~\ref{sec-outline} that we will be especially
interested in cases where $m$ and $n$ range over intervals of size
close to $q^{1/2}$. Our results in this section go in the direction of the conditional
estimate in Proposition~\ref{bilinearpr}, being of similar (or better)
quality for special coefficients $(\alpha_m)$ and $(\beta_n)$. 
Proposition~\ref{bilinearpr} itself will be proved in Subsection~\ref{sec-conj}, depending
on a conjecture on certain complete sums over finite fields.

\begin{theorem}\label{CombinedTheorem}
  Let $q$ be a prime number, let $a$ be an integer coprime with $q$, $M, N\geq 1$,  
  $\mcN$ an interval of length $N$, and
  $(\alpha_m)_m$, $(\beta_n)_n$ two sequences supported respectively on $[1,M]$ and $\mcN$.
\par
\emph{(1)} If $M, N \leq q$, we have
\begin{equation}\label{typeIIgen}\sumsum_{m\leq M,\,n\in\mcN}{\alpha_m}\beta_n \Kl(amn;q)\ll (qMN)^{\eps}(MN)^{1/2} \| \alpha \|_2 \| \beta \|_2 \bigl(M^{-1/2}+q^{1/4}N^{-1/2}\bigr)
\end{equation} 
for any $\eps>0$. 
\par
\emph{(2)} If  the conditions
\begin{equation}\label{MNcond}
M, N \leq q, \quad  MN\leq q^{3/2},\quad M\leq N^2,
\end{equation}
are satisfied, then we have
\begin{equation}\label{typeIgen}
  \sumsum_{m\leq M,\,n\in\mcN}\alpha_m\Kl(amn;q)\ll (qMN)^{\eps} (\|\alpha\|_1\|\alpha\|_2)^{1/2}M^{1/4}N\bigl(q^{\frac14}M^{-\frac1{6}}N^{-\frac{5}{12}}). 
\end{equation}
\par
\emph{(3)} Let $W_i$, for $1\leq i\leq 2$, be smooth, compactly
supported functions satisfying
\begin{equation}\label{Wbound2}
  W_i^{(j)}(x)\ll_j Q^j,\ i=1,2, 
\end{equation}
for some $Q\geq 1$ and for all $j \geq 0$.  There exists an absolute
constant $A\geq 0$ such that for any $\eps>0$, we have
\begin{equation}\label{eq-fkm1}
 \sum_m \sum_n W_1\Bigl(\frac{m}{M}\Bigr) W_2\Bigl(\frac{n}{N}\Bigr)
  \Kl(amn;q)\ll_{\eps} q^\eps Q^A
  MN\Bigl(\frac1{q^{1/8}}+\frac{q^{3/8}}{(MN)^{1/2}}\Bigr).
\end{equation}
All bounds are uniform in $a$, and we write as usual
\[ \| \alpha \|_1 = \sum_m |\alpha_m|,\quad \| \alpha \|_2 = \Bigl(\sum_m |\alpha_m|^2\Bigr)^{1/2}. \]
\end{theorem}

The first part~(\ref{typeIIgen}) and the third~(\ref{eq-fkm1}) have
been proven by Fouvry, Kowalski and Michel in Theorems 1.17 and 1.16
(respectively) of~\cite{FKM2} (building on~\cite{FKM1}), in
considerably greater generality. Thus it remains to prove the second
part. 

\subsection{General setup}

Some of our arguments are valid for bilinear forms involving a more
general kernel $K(mn)$ modulo $q$ than the Kloosterman sums
$\Kl(amn;q)$. It is therefore useful to consider first the general
problem of bounding a general ``type I'' sum
\begin{equation}
\label{TypeISumGeneral}
B(K,\bfalpha_M,1_\mcN):=\sum_{m\leq M}\alpha_m\sum_{n \in \mathcal{N}} K(mn),
\end{equation}
where $K:\Fq\ra \Cc$ is an arbitrary function. We assume that
$|K(x)| \ll 1$ with an absolute implied constant. \par

The proof is a slight generalization of the method of \cite{AnnENS}: given $A,B\geq 1$ such that
\begin{equation}\label{ABbounds}
AB\leq N,\ 2AM<q, 
\end{equation}
 we have
\begin{align*}
B(K,\bfalpha_M,1_\mcN)&=\frac{1}{AB}\sumsum_{\substack{A < a \leq 2A\\ B < b \leq 2B}}\sum_{m\leq M}\alpha_m\sum_{n+ab\in\mcN} K(m(n+ab))\\
&=\frac{1}{AB}\sumsum_{\substack{A < a \leq 2A\\ B < b \leq 2B}}\sum_{m\leq M}\alpha_m\sum_{n+ab\in\mcN} K(am(\ov an+b)),
\end{align*}
where, as usual, $a\ov a\equiv 1\mods q$. 

By the method of \cite{AnnENS}*{p.\ 116}, we get
\[ B(K,\bfalpha_M,1_\mcN)\ll \frac{\log q}{AB}\sumsum_{r\mods q,\,s\leq
 2 AM}\nu(r,s)\Bigl|\sum_{B < b \leq 2B }\eta_bK(s(r+b))\Bigr| \]
for
\[ \nu(r,s)=\sumsumsum_\stacksum{A < a \leq 2A,\ m\leq M,\ n\in\mcN',}{am=s,\ \ov
an\equiv r\mods q}|\alpha_m|,  \]
where $\mcN'\supset\mcN$ is some interval of length $2N$ and $(\eta_b)_{B < b \leq 2B}$ are some complex numbers such that
$|\eta_b|\leq 1$.  We have the bounds
\[ \sum_{r,s}\nu(r,s)\ll AN\sum_{m\leq M}|\alpha_m| \]
and
\begin{eqnarray*}
\sum_{r,s}\nu(r,s)^2&=&\multsum_\stacksum{a,m,n,a',m',n'}{\stacksum{am=a'm',}{a'n=an'\mods q}}|\alpha_m||\alpha_{m'}|\ll \sum_{a,m}|\alpha_m|^2 \multsum_\stacksum{n,a',m',n'}{\stacksum{am=a'm',}{a'n=an'\mods q}}1\ll  q^\eps AN\sum_{m}|\alpha_m|^2.
\end{eqnarray*}
Here, we have used the inequality $|\alpha_m||\alpha_{m'}|\leq |\alpha_m|^2+|\alpha_{m'}|^2$ and the fact that, once $a$ and $m$ are given, the equation $am=a'm'$ determines $a'$ and $m'$ up to $O(q^\eps)$ possibilities, and, for each such $a',m'$ and each $n\in\mcN'$, the congruence $a'n=an'\mods q$ has at most two solutions in the interval $\mcN'$ since it has length $\leq 2q$ (cf.~\cite{AnnENS}*{p.\ 116}).

From these bounds and from H\"older's inequality, we obtain that
\begin{equation}\label{obtain}
AB\cdot B(K,\bfalpha_M,1_\mcN)\ll q^\eps (AN)^{3/4}(\|\alpha\|_1\|\alpha\|_2)^{1/2}\Bigl(\sumsum_{r\mods q,1\leq s\leq AM}\bigl|\sum_{B < b\leq 2B}\eta_b K(s(r+b))\bigr|^4\Bigr)^{1/4}.
\end{equation}
\par
Expanding the fourth power, the inner term of the second factor can be
written as
\[ \sum_{\bfb\in\mcB}\eta(\bfb)\Sigma(K, \bfb), \]
where $\mcB$ denotes the set of quadruples $\bfb=(b_1,b_2,b'_1,b'_2)$
of integers satisfying $B < b_i,b'_i \leq 2B$ ($i=1,2$),  the coefficients
$\eta(\bfb)$ satisfy $|\eta(\bfb)|=1$ for all $\bfb\in\mcB$, and we
denote
\begin{equation}\label{eq-def-sigma}
\Sigma(K,\bfb)=\sumsum_\stacksum{r\mods
  q}{1\leq s\leq AM}K(s(r+b_1))K(s(r+b_2))\ov{K(s(r+b'_1))K(s(r+b'_2))}.
\end{equation}

Let $\mcB^{\Delta}$ be the subset of $\bfb$ admitting a subset of two
entries matching the entries of the complement (for instance such that
$b_1=b'_1$ and $b_2=b'_2$); one has $|\mcB^{\Delta}|=O(B^2)$. For such
$\bfb$, we use the trivial bound for $\Sigma(K, \bfb)$, getting
\begin{equation}\label{bDelta}
  \sum_{\bfb\in\mcB^{\Delta}}|\Sigma(K, \bfb)|\ll AB^2Mq,
\end{equation}
where the implied constant depends only on $H$.

To bound the contribution of the $\bfb\not\in \mcB^{\Delta}$,
we complete the $s$-sum using additive characters and obtain 
\[ \Sigma(K,\bfb)\ll (\log q)\max_{h\mods q}|\Sigma(K,\bfb,h;q)|, \]
where 
\begin{equation}
\label{ResultOfCompletion}
\Sigma(K,\bfb,h):=\sumsum_{r,s\mods
  q}K(s(r+b_1))K(s(r+b_2))\ov{K(s(r+b'_1))K(s(r+b'_2))}e_q(hs).
\end{equation}

The procedure we have described gives a general scheme for estimating
special bilinear forms \eqref{TypeISumGeneral} with a general
uniformly bounded kernel $K$: the sum $B(K,\bfalpha_M,1_\mcN)$ is estimated as in
\eqref{obtain}, where the contribution of the diagonal quadruples
$\bfb\in\mcB^{\Delta}$ is bounded in \eqref{bDelta}, and the
contributions of off-diagonal quadruples $\bfb\not\in\mcB^{\Delta}$
are estimated in terms of the \emph{complete} sums
$\Sigma(K,\bfb,h)$ given by \eqref{ResultOfCompletion}.

If we now insert the trivial bound $\Sigma(K,\bfb,h)\ll q^2$, we
obtain a bound that is never better than the trivial estimate
$B(K,\bfalpha_M,1_\mcN)\ll MN$. We must improve on this by exhibiting cancellation
in the complete sum $\Sigma(K,\bfb,h)$ by exploiting the structure
of $K$.

In Section~\ref{CompleteSumsSubsection}, we will further show how, for
kernels $K$ that are themselves given by a complete exponential sum of
a specific shape, the estimation of $\Sigma(K,\bfb,h)$ reduces to
the estimation (with square-root cancellation) of certain auxiliary
additive character sums in two variables, which can in turn sometimes be
treated using the Riemann Hypothesis over finite fields.

In particular, we will prove:

\begin{proposition}\label{pr-nondiag}
  Let $q$ be a prime and define $K(a)=\Kl(a;q)$. With
  notation as above, for all $\bfb\in\mcB-\mcB^{\Delta}$ and all
  $h\in\Ff_q$, we have
\begin{equation}\label{onehas}
|\Sigma(K,\bfb,h)|\ll q,
\end{equation}
where the implied constant is absolute.
\end{proposition}

Combining this bound and the contribution from $\mcB^\Delta$,
we obtain (in the case of Kloosterman sums) by \eqref{obtain},
\eqref{bDelta}, and \eqref{onehas} that
\[ B(K,\bfalpha_M,1_\mcN)\ll q^\eps (AB)^{-1}(AN)^{3/4}(\|\alpha\|_1\|\alpha\|_2)^{1/2}(AB^2Mq+B^4q)^{1/4}. \]
We then finish the proof of the second part of
Theorem~\ref{CombinedTheorem} by choosing
\[ A=\frac 12 M^{-\frac13}N^{\frac23},\ B=(MN)^{\frac13}; \]
note that the conditions \eqref{ABbounds} as well as $A,B\geq 1$ are
satisfied by \eqref{MNcond}.  \qed

\begin{remark}
  The estimate \eqref{onehas} achieves square-root cancellation in the
  two-variable complete sum $\Sigma(K,\bfb,h)$. A weaker bound
  $\Sigma(K,\bfb,h)\ll q^{3/2}$ can be proved easily and quite
  generally directly (by fixing one of the variables) from, say,
  \cite{FKMSP}, but this yields a power saving over the trivial bound
  for $B(K,\bfalpha_M,1_\mcN)$ (say, if $\alpha_m=1$) only if $N>q^{1/2+\delta}$ and
  $MN^{5/2}>q^{2+\delta}$ for some $\delta>0$. This shows that we do
  require a stronger bound in the critical range $M,N\asymp q^{1/2}$.
\end{remark}

\subsection{Reduction to two-variable character sums}
\label{CompleteSumsSubsection}

We will now study the sums $\Sigma(K,\bfb)$ for special kernels $K$.
Precisely, we assume that there exists a rational function $f\in
\Fq(T)$, not a linear polynomial, such that
\[ K(x)=q^{-1/2}\sums_{u \mods q} e_q(f(u)+xu), \]
where the asterisk denotes that the values of $u$ where $f$ has a pole
are excluded. In that case, Weil's theory shows that $K$ is bounded by
some constant $H$ depending only on the degrees of the numerator and
denominator of $f$, and we can attempt to estimate the corresponding
bilinear form as in the previous section.

Replacing $K$ in~(\ref{eq-def-sigma}) by this formula and performing
the averaging over $s$, we obtain
\[ \Sigma(K, \bfb,h)=q^{-1}\sum_{r\mods q}\sums_{(u,v,u',v')\in
  V_r(\Fq)}e_q(f(u)+f(v)-f(u')-f(v')),  \]
where $V_r(\Fq)$ is set of solutions $(u,v,u',v')$ of the equation
\begin{equation*}\label{satisfies}
r(u+v-u'-v')+b_1u+b_2v-(b'_1u'+b'_2v')+h=0.
\end{equation*}
\par
The sum further decomposes into two sums, depending on whether
$(u,v,u', v')$ satisfies the additional equation
\[ u+v-u'-v'=0 \]
or not.  If $u+v-u'-v'\not=0$, there is, for a given $(u,v,u',v')$,
only one possible $r$ such that $(u,v,u',v')\in V_r(\Fq)$, and
therefore the contribution $\Sigma_1$ of these terms  to $\Sigma(K,
\bfb,h;q)$ of these terms is equal to
\begin{displaymath}
\begin{split}  \Sigma_1&=q^{-1}
\sums_\stacksum{u,v,u',v' \mods
    q}{u+v-u'-v'\not=0}e_q\bigl(f(u)+f(v)-f(u')-f(v')\bigr) = q^{-1}\Bigl(q^2|K(0)|^4-q\sum_{r \mods q} |K(r)|^4\Bigr)\ll q.
\end{split}
\end{displaymath}
\par
We are left with the following $2$-dimensional exponential sum over $\Fq$
\[ S(f,h,\bfb):=\sumsum_{(u,v,u',v')\in W(\Fq)} e_q(f(u)+f(v)-f(u')-f(v')),  \]
where $W(\Fq)$ is the set of quadruples $(u,v,u',v')\in\Ff_q^4$
satisfying
\begin{equation*}\label{2dim}
\begin{cases}
  u+v=u'+v',\\
  b_1u+b_2v=b'_1u'+b'_2v'-h.
\end{cases}
\end{equation*}

We will prove the following estimate for these sums:

\begin{theorem}\label{thm1/T} 
  Let $f(T)=1/T\in\Fq(T)$ and consider four non-zero linear forms in
  two variables
\[ l_1(u,v):=u,\ l_2(u,v):=v,\ l_3(u,v)=\alpha u+\beta v,\ l_4(u,v)=\gamma u+\delta v. \]
If
\begin{equation*}\label{1/Thypo}
 \{l_3,l_4\}\not=\{l_1,l_2\},
\end{equation*}
then, for all $h\in\Fq$, we have
\[ \sum_{u,v}e_q\bigl(f(u)+f(v)-f(l_3(u,v)+h)-f(l_4(u,v)-h)\bigr)\ll q, \]
where the implied constant is absolute.
\end{theorem}

We will prove this in the next section. Assuming the result, we
conclude the proof of Proposition \ref{pr-nondiag} (and hence of Theorem~\ref{CombinedTheorem})  as follows: if
$b'_1\not=b'_2$, we can write
\[ S(f,h,\bfb)=\sum_{u,v}e_q\bigl(f(u)+f(v)-f(l_3(u,v)+h')-f(l_4(u,v)-h')\bigr),  \]
where
\begin{gather*}
  l_3(u,v)=\frac{b_1-b'_2}{b'_1-b'_2}u+\frac{b_2-b'_2}{b'_1-b'_2}v,
\quad\quad
  l_4(u,v)=\frac{b'_1-b_1}{b'_1-b'_2}u+\frac{b'_1-b_2}{b'_1-b'_2}v,\\
  h'=\frac{h}{b'_1-b'_2}.
\end{gather*}
\par
Simple checks show that the sets $\{l_1,l_2\}$ and $\{l_3,l_4\}$ thus
defined coincide only if $\bfb\in\mcB^{\Delta}$. Hence we then get
\[ S(f,h,\bfb)\ll q \]
for all $\bfb\in \mcB-\mcB^{\Delta}$ and all $h\in\Fq$.
\par
If $b'_1=b'_2$ but $b_1\not=b_2$, we can proceed in a similar way,
exchanging the roles of $(u,v)$ and $(u',v')$. This gives the desired
bounds except when $b_1=b_2$ and $b'_1=b'_2$. But such quadruples
$\bfb$ are also in $\mcB^{\Delta}$. \qed

\subsection{Estimate of two-variable character sums}

We prove Theorem~\ref{thm1/T} in this section. By a general criterion
(due to Hooley~\cite{Hoo}*{Theorem 5} and Katz~\cite{Sommes}*{Cor.\ 4},
see also~\cite{AnnENS}*{Prop. 2.1}), the desired estimate follows from
the Riemann Hypothesis over finite fields of Deligne for any $h\in\Fq$
such that the rational function
\[ F(U,V)=f(U)+f(V)-f(l_3(U,V)+h)-f(l_4(U,V)-h)\in\Fq(U,V) \]
is \emph{not composed}, which means that it is not of the shape
\[ F=Q\circ P,  \]
where $P\in\bFq(U,V)$ and where $Q\in\bFq(T)$ is a rational function
which is not a fractional linear transformation $(aT+b)/(cT+d)$ (in
particular $F(U,V)$ is not constant).
\par
This is a purely geometric question and we will show this more
generally for $h,\alpha,\beta,\gamma,\delta$ in $\bFq$, under the
assumption that $\{l_1,l_2\}\not=\{l_3,l_4\}$. 
\par
In the following, we denote by $C\in\bFq$ a non--zero constant, the
value of which may change from one line to another. We follow closely
the method of \cite{AnnENS}*{Proposition 2.3} but first we make the
birational change of variables
\[ X=U/V,\ Y=V, \]
so that 
\begin{align*}
 F(U,V)& =f(XY)+f(Y)-f(Yl_3(X,1)+h)-f(Yl_4(X,1)-h)\\
& =\frac{1}{XY}+\frac{1}{Y}-\frac{1}{Yl_3(X,1)+h}-\frac{1}{Yl_4(X,1)-h}.
\end{align*}
\par
We then need to prove that $F(XY,Y)$ is not of the shape
\[ \frac{Q_1(P_1(X,Y)/P_2(X,Y))}{Q_2(P_1(X,Y)/P_2(X,Y))}, \]
where $P_1(X,Y), P_2(X,Y)\in\bFq[X,Y]$ are coprime polynomials in two
variables and
\[ Q_1(T)=C\prod_{\lambda}(T-\lambda)^{m(\lambda)}, \ Q_2(T)=\prod_{\mu}(T-\mu)^{m(\mu)} \]
are coprime polynomials in one variable (here $m(\lambda)$ and
$m(\mu)$ denote the multiplicity of the zeros $\lambda$ and
$\mu$). Moreover, up to changing the variable $T$ by a fractional
linear transformation, we may assume that the degrees
$q_1(=\sum_\lambda m(\lambda))$ and $q_2(=\sum_\mu m( \mu))$ of $Q_1$
and $Q_2$ satisfy the inequality
 \begin{equation}\label{q1>q2}
 q_1>q_2,
 \end{equation} which means that  $\infty$ is a pole of $Q$.  Our objective is then to show that  
 \begin{equation}\label{q1=1}
 q_1=1.
 \end{equation}
    We have the identity 
\[ F(XY,Y)=\frac{C\prod_{\lambda}(P_1(X,Y)-\lambda P_2(X,Y))^{m(\lambda)}}{P_2(X,Y)^{q_1-q_2}\prod_{\mu}(P_1(X,Y)-\mu P_2(X,Y))^{m(\mu)}}=:\frac{{\rm NUM}(X,Y)}{{\rm DEN}(X,Y)}. \]
In this latter expression, the numerator and denominator, ${\rm NUM}(X,Y)$ and ${\rm DEN}(X,Y)$, are coprime.  We also  have  
 \begin{equation}\label{alsohave}
 F(XY,Y)=\frac{(Yl_3(X,1)+h)(Yl_4(X,1)-h)(1+X)-XY^2(l_3(X,1)+l_4(X,1))}{XY(Yl_3(X,1)+h)(Yl_4(X,1)-h)}.
 \end{equation}
 By the assumption \eqref{q1>q2}, we deduce that $P_2 (X,Y)$ is not a constant polynomial (it suffices to compare the differences of the total degrees of the numerator and of the denominator of the two above expressions of $F(XY,Y)$).
We distinguish two cases to finish the proof.
\par
(1) Assume first that $h\not=0$.  If $l_3+l_4\not=0$ then the
numerator and denominator of \eqref{alsohave} are coprime and are
equal to $C\cdot {\rm NUM}(X,Y)$ and $C\cdot {\rm DEN}(X,Y)$
respectively. Since the factors $X$, $Y$, $Yl_3(X,1)+h$, $Yl_4(X,1)-h$
are simple and coprime and since $P_2 (X,Y)$ is not constant, we have
\begin{equation}\label{q1-q2}
q_1-q_2=1,
\end{equation}
 and if $q_2  \not=0$, we necessarily have $m(\mu)=1$ for any $\mu$.

In particular if $q_2=0$, we obtain \eqref{q1=1} and we are done.

Suppose now that  $q_2\geq 1$.  If  $Y$ does not divide $P_2$, it divides some $P_1-\mu P_2$ (and then $m(\mu)=1$),  and up to   the change of variable
$T\mapsto T+\mu$ (which does not change the condition $q_1-q_2>0$) we may assume that $\mu=0$: hence  $Y\mid P_1$
and all the zeros $\lambda$ of $Q_1$ are non-zero. Hence, in all the cases, we have  $Y\mid P_1P_2$,  from which   we deduce  the equality
\[ {\rm NUM}(X,0)=CP_1(X,0)^{q_1}\hbox{ or }CP_2(X,0)^{q_1}, \]
but ${\rm NUM}(X,0)=-h^2(1+X)$ and therefore $q_1=1$. This contradicts the equality \eqref{q1-q2} and the assumption $q_{2}\geq 1$. 

The proof when $l_3+l_4=0$ is identical except that the fraction $F(XY,Y)$ simplifies to the reduced fraction 
\[ F(XY,Y)=\frac{1+X}{XY}. \]
\par
(2) Assume now finally that $h=0$.
\label{Caseh0}
In this case we have
\begin{equation}\label{we have}
F(XY,Y)=\frac{l_3(X,1)l_4(X,1)(1+X)-X(l_3(X,1)+l_4(X,1))}{XYl_3(X,1)l_4(X,1)} . 
\end{equation}
Let us assume that $F(XY,Y)\not=0$. The polynomials ${\rm NUM}(X,Y)$
and ${\rm DEN}(X,Y)$ divide the numerator and denominator of the
right-hand side of \eqref{we have}, in particular ${\rm NUM}(X,Y)$
does not depend on $Y$. Suppose that $q_2\geq 1$; as above
(possibly up to a change of variable $T\mapsto T+\mu $), we
may assume that $0\not\in\{\lambda \mid Q_1(\lambda)=0\}$ and that
either $Y$ divides $P_1(X,Y)$ or $P_2(X,Y)$ (but not both); in either
cases, this is not compatible with the equality
\[ {\rm NUM}(X,Y)=C\,\prod_{\lambda}(P_1(X,Y)-\lambda P_2(X,Y))^{m(\lambda)}, \]
since the left-hand side only depends on $X$ and the $\lambda$ are $\not=0$. Therefore $q_2=0$ and $Y$ divides $P_2(X,Y)^{q_1}$ to order $1$ so that $q_1=1$.
The only remaining case is when
\[ l_3(X,1)l_4(X,1)(1+X)-X(l_3(X,1)+l_4(X,1))=0. \]
By the explicit expressions $l_{3}(X,1) =\alpha X +\beta$ and $l_{4}(X,1) =\gamma X +\delta $ and by the fact that $l_{3}$ and $l_{4}$ are not zero,  the above equality  is equivalent to 
$\{l_3,l_4\}=\{l_1,l_2\}.$
\qed

\subsection{Conjectural bounds for bilinear forms in Kloosterman sums}
\label{sec-conj}

In this section, which is not needed for the proof of the
unconditional results of this paper, we establish the following proposition concerning bilinear sums of Kloosterman sums conditionally on a square root cancellation bound for a certain complete sum of products of Kloosterman sums in three variables, which we state as Conjecture \ref{completesumconj}.

\begin{proposition}[Bilinear forms of Kloosterman
  sums]\label{bilinearpr} Let $q$ be a prime, $(a, q) = 1$,
  $\mcN\subset\Rr$ an interval of length $N$ and
  $(\alpha_m)_m$, $(\beta_n)_n$ be two sequences of complex numbers supported respectively on $[1,M]$ and $\mcN$ and with $\ell_2$-norms given by
\[ \|\alpha\|_2^2=\sum_{m\leq M}|\alpha_m|^2,\ \|\beta\|_2^2=\sum_{n\in\mcN}|\beta_n|^2. \]
Assuming that 
\begin{equation}\label{MNbounds2}1\leq M,N\leq q,\ q^{\frac{1}{4}}\leq MN\leq q^{\frac{5}{4}}\hbox{ and } M\leq q^{\frac{1}{4}}N,	
\end{equation}
and that \emph{Conjecture \ref{completesumconj}} holds,   one has
\begin{equation*}\label{eq-typeIIconj}
\sumsum_{m\leq M,n\in\mcN}{\alpha_m}\beta_n\Kl(amn;q)\ll q^{\eps}\|\alpha\|_2\|\beta\|_2(MN)^{1/2}\bigl(M^{-\frac1{2}}+q^{\frac{11}{64}}(MN)^{-\frac{3}{16}}\bigr)
\end{equation*} 
for any $\eps>0$, uniformly in $a$. 
\end{proposition}

\begin{remark} The main reason to believe that the above bound should hold unconditionally is the unconditional bound obtained by Fouvry and Michel in
\cite{AnnENS}*{\S VII}, for bilinear forms of the shape
\begin{equation}\label{boundfor}
B(K,\bfalpha_M,\bfbeta_N)=\sumsum_{m\leq M,n\leq N}\alpha_m\beta_n K(mn)		
\end{equation}
for kernels $K(x)$ of the shape
\[ K(x)=e\Bigl(\frac{x^k+x}{q}\Bigr) \]
for some fixed integer $k\not=0,1,2$. They obtained the bound \[ \sumsum_{\substack{M < m \leq 2M\\ N < n \leq 2N}}\alpha_m\beta_n K(mn)\ll_{\eps,k} q^\eps
\|\alpha\|_2\|\beta\|_2(MN)^{1/2}(M^{-1/2}+q^{\frac{11}{64}}(MN)^{-\frac{3}{16}}) \]
for any $\eps>0$ (actually  a slightly
  weaker bound with $\|\alpha\|_2\|\beta\|_2$ replaced by $(MN)^{1/2}$
  under the assumption that $|\alpha_m|,|\beta_n|$ are bounded by $1$;
  as we show below, the method of \cite{AnnENS} yields the slightly
  stronger bound presented here.) This bound was ultimately a consequence of bounds for families of multivariable complete algebraic exponential sums which where obtained using the work of Deligne and Katz.
\end{remark}
\par 
We establish Proposition~\ref{bilinearpr} by repeating the argument of Fouvry
and Michel. With $K(x)=\Kl(ax;q)$, the Cauchy--Schwarz inequality gives
\[ \Bigl|\sumsum_{m\leq M,\ n\in\mcN}\alpha_m\beta_n K(mn)\Bigr|^2 \leq
\|\beta\|_2^2\sumsum_{m_1,m_2\leq
  M}\alpha_{m_1}\ov\alpha_{m_2}\sum_{n\in\mcN}K(m_1n)\ov
K(m_2n)=:\|\beta\|_2^2(\Sigma^=+\Sigma^{\not=}),  \]
where $\Sigma^=$ is the contribution of the diagonal terms $m_1\equiv
m_2\mods q$ and $\Sigma^{\not=}$ is the remaining off-diagonal
contribution. The diagonal term is bounded by $O(\|\alpha\|_2^2N)$. For the
remaining terms we apply again Vinogradov's ``shift by $ab$'' trick:
given $A,B\geq 1$ satisfying the conditions \eqref{ABbounds} (these
will be satisfied by \eqref{MNbounds2} after a suitable choice of
$A,B$), the off-diagonal term $\Sigma^{\not=}$ is bounded by
\[ \frac{1}{AB}\sumsum_{\substack{A < a \leq 2A\\ B < b \leq 2B}}\sumsum_\stacksum{m_1,m_2\leq
  M}{m_1\not\equiv m_2\mods
  q}\alpha_{m_1}\ov\alpha_{m_2}\sum_{n+ab\in\mcN}K(am_1(\ov an+b))\ov
K(am_2(\ov an+b)), \]
which is itself bounded by
\[ \ll\frac{q^\eps}{AB}\sumsum_\stacksum{r\mods q,\ 1\leq s_1,s_2\leq
  AM}{s_1\not\equiv s_2\mods q}\nu(r,s_1,s_2)\Bigl|\sum_{B < b \leq 
  2B}\eta_bK(s_1(r+b))\ov K(s_2(r+b))\Bigr|, \]
where $|\eta_b|\leq 1$ and
\[ \nu(r,s_1,s_2)=\multsum_\stacksum{A < a \leq 2A,\ m_1,m_2\leq M,\ n\in\mcN',}{am_1=s_1,\ am_2=s_2,\ \ov an\equiv r\mods q}|\alpha_{m_1}||\alpha_{m_2}|  \]
with $\mcN' \supset \mcN$ as before. 
\par
By the same reasoning as above we have
\[ \sumsumsum_{r,s_1,s_2}\nu(r,s_1,s_2)\ll AN\|\alpha\|_1^2\leq AMN\|\alpha\|_2^2 \]
and
\[ \sumsumsum_{r,s_1,s_2}\nu(r,s_1,s_2)^2\ll q^\eps AN\|\alpha\|^4_2. \]
From these bounds and H\"older's inequality  and \eqref{ABbounds}, we obtain as in  \cite{AnnENS}*{Lemma
7.1} that $\Sigma^{\not=}$ is
bounded by
\begin{multline*}
  \frac{q^\eps}{AB}(AN)^{3/4}M^{1/2}\|\alpha\|^2_2\\
  \Bigl(\sum_{\bfb}\Bigl|\sum_{r\mods q}\sumsum_\stacksum{1\leq s_1,s_2\leq
    AM}{s_1\not\equiv s_2\mods q}\prod_{i=1}^2K(s_1(r+b_i))\ov
  K(s_2(r+b_i))\ov{K(s_1(r+b_{i+2}))\ov
    K(s_2(r+b_{i+2})}\Bigr|\Bigr)^{1/4}
\end{multline*}
for $\bfb$ running over the set $\mcB$ of quadruples
$(b_1,b_2,b_3,b_4)$ satisfying $B < b_i  \leq 2 B$. We bound the inner triple
sum over $r, s_1, s_2$ depending on the value taken by $\bfb$: let
\[ \mcB^\Delta\subset\mcB \]
be the ``diagonal" set of elements $\bfb$ for
which some pair $(b_i,b_j)$ ($i,j\leq 4$) with distinct indices equals
a pair having complementary indices (for instance
$(b_1,b_4)=(b_3,b_2)$). We have $|\mcB^\Delta|=O(B^2)$ and for
$\bfb\in \mathcal{B}^\Delta$ we use the trivial bound to obtain
\[ \sum_{\bfb\in \mathcal{B}^\Delta}\Bigl|\sum_{r\mods q}\sumsum_{s_1,s_2}\cdots \Bigr|\ll qA^2B^2M^2. \]
For the $O(B^4)$ elements not contained in $\mathcal{B}^{\Delta}$ we
detect the condition $s_1\not\equiv s_2\mods q$ via additive
characters, writing
\[ \delta(s_1\not\equiv s_2\mods q)=1-\frac1q\sum_{\lambda\mods
  q}e_q(\lambda(s_1-s_2)). \]
We then complete the $s_1$, $s_2$ sums, also using additive
characters: for $\lambda,\mu_1,\mu_2\in\Zz/q\Zz$ let
\[  \mathcal{S}(r,\lambda;q)=\sum_{s\mods q}K(s(r+b_1))K(s(r+b_2))\ov{K(s(r+b_3))K(s(r+b_4))}e_q(\lambda s), \]
\[ \mathcal{R}(\mu_1,\mu_2;q)=\sum_{r\mods q}\mathcal{S}(r,\mu_1;q)\ov{\mathcal{S}(r,\mu_2;q)}, \]
and
\[ \Sigma(\bfb,\mu_1,\mu_2;q)=\mathcal{R}(\mu_1,\mu_2;q)-\frac{1}q\sum_{\lambda\mods q}
\mathcal{R}(\mu_1+\lambda,\mu_2+\lambda;q). \]

\subsection{Correlation sums}
We now formulate a conjectural bound on the sum $\Sigma(\bfb,\mu_1,\mu_2;q)$. To motivate this conjecture, let us briefly examine the structure and the significance of the sums $\mathcal{S},\mathcal{R}$ and $\Sigma$.

Given $r,\lambda$ and $\bfb$, the sum $\mathcal{S}(r,\lambda;q)$ is a one-variable sum of a product of the four Kloosterman sums 
$s\mapsto K(s(r+b_i))$,  $i=1,\ldots,4$, and the additive phase $e_q(\lambda s)$. It is well known that as $s$ varies, the Kloosterman sums  oscillate rather wildly and moreover, for distinct values of the $b_i$'s, these oscillations are independent of each other, so that typically square-root cancellation occurs:
\[ \mathcal{S}(r,\lambda;q)=O(q^{1/2}). \]
For this and more general sums of that type, we refer to the article~\cite{FKMSP}
which builds crucially on the works of Deligne and Katz \cites{WeilII,GKM,ESDE}. The sum $\mathcal{R}(\mu_1,\mu_2;q)$ deals with the variation of the sums $\mathcal{S}(r,\lambda;q)$; more precisely it measures to which extent the functions 
$r\mapsto \mathcal{S}(r,\mu_1;q)$ and $r\mapsto \mathcal{S}(r,\mu_2;q)$ correlate. If there is no correlation, it is then natural to
 expect from Deligne's formalism of weights that square-root cancellation occurs again, and so 
 \[ \mathcal{R}(\mu_1,\mu_2;q)=O(q^{3/2}). \]
 On the other hand, when the sums \emph{do} correlate, one expect $\mathcal{R}(\mu_1,\mu_2;q)$ to be the sum of a main term of size $q^2$ and of an error term, more precisely
 \[ \mathcal{R}(\mu_1,\mu_2;q)=q^2+O(q^{3/2}). \]
 This is essentially the content of the conjecture below, which also incorporates the correlation and non-correlation cases; see \cite{KMS} for further discussions on this conjecture.

\begin{conjecture}\label{completesumconj}
  There exists a constant $C$ such that for any prime $q$, every
  integer $a$ coprime with $q$, every $\mu_1,\mu_2\in\Fq$ and every
  $\bfb\in \mathcal{B}^{gen} := \mathcal{B} \setminus
  \mathcal{B}^{\Delta}$ we have
\[ |\Sigma(\bfb,\mu_1,\mu_2;q)|\leq C q^{3/2}; \]
here the sum $\Sigma$ is the sum relative to the function $K(x)=\Kl(a x;q)$.
\end{conjecture}

If we assume Conjecture \ref{completesumconj}, we obtain that
\[ \sum_{\bfb\in \mathcal{B}^{gen}}\Bigl|\sum_{r\mods
  q}\sumsum_{s_1,s_2}\cdots\Bigr|\ll q^\eps B^4q^{3/2}. \]
Hence, under this assumption, we have
\[ \Sigma^{\not=}\ll\frac{q^\eps}{AB}(AN)^{3/4}M^{1/2}\|\alpha\|^2_2(A^2B^2M^2q+B^4q^{3/2})^{1/4}. \]
We may choose (see \cite{AnnENS}*{p.128})
\[ A=q^{1/8}M^{-1/2}N^{1/2},\ B=q^{-1/8}(MN)^{1/2}, \]
for which \eqref{ABbounds} as well as $A,B\geq 1$ are satisfied by
\eqref{MNbounds2}. Combining this bound with that for $\Sigma^{=}$, we
conclude that Proposition \ref{bilinearpr} follows from Conjecture
\ref{completesumconj}.\qed

\section{Evaluation of moments of \texorpdfstring{$L$-functions}{L-functions}}\label{sec5}

In this section, we implement the strategy sketched in
Section~\ref{sec-outline} to prove Theorems \ref{improvedyoung},
\ref{mixedthm} and \ref{cuspthm}.

\subsection{First steps}

Let $f,g$ be either Hecke cusp forms of level $1$, or the
Eisenstein series $E$ defined in \eqref{Eisendef}. Let $q$ be a prime number. We decompose the second moment \eqref{SecondMomentDisplayed}
into the moments of twists by even and odd characters separately:
\[ M_{f,g}(q)=M_{f,g,1}(q)+M_{f,g,-1}(q), \]
where, for $\sigma\in \{-1,1\}$, we put
\[ M_{f,g,\sigma}(q)= \frac{1}{\varphi^*(q)}\sum_{\substack{\chi(-1)=\sigma\\
    \chi \text{ primitive}}}L(f \otimes\chi,1/2)L(g\otimes\ov\chi,1/2). \]
\par
Using the computation of the root number in
Lemma~\ref{propfcteqn} and the invariance of the parity
$\chi(-1)$ under
complex conjugation, we find that
\[ M_{f,g, \sigma }(q)=\frac{1+\eps(f,g,\sigma)}{2}M_{f,g,\sigma}(q), \]
where $\eps(f,g,\sigma)$ is the root number $\eps(f,g,\chi)$ for any
primitive character $\chi$ with parity $\chi(-1) = \sigma$. Thus
$M_{f,g,\sigma}(q)=0$ unless
\[ \eps(f,g,\sigma)=1, \]
which we henceforth assume. By the approximate functional equation \eqref{fcteqn}, we have
\[ L(f \otimes \chi, 1/2) \overline{L(g \otimes \chi, 1/2)} =
2\sum_{m,n\geqslant 1} \frac{\lambda_f(m) \lamg(n)  }{(mn)^{1/2}}\chi(m)
\bar{\chi}(n) V_{f,g,\sigma}\left(\frac{mn}{q^2}\right), \]
where   the function
$V_{f,g,\sigma}$ is given by~(\ref{Vfgdef}).

We now average over $\chi$ of parity $\sigma$. The orthogonality
relation for these characters is
\begin{equation}\label{eq-ortho}
  \frac{2}{q-1}\sum_\stacksum{\chi\mods q}{\chi(-1)=\sigma}
  \chi(m)\ov{\chi(n)}=\delta_{m\equiv n\mods q}+\sigma\delta_{m\equiv - n\mods q}
\end{equation}
for $q$ prime and any integers $m$ and $n$ such that
$(mn,q)=1$. Inserted in the above formula, it yields
\begin{equation*}\label{decomp}
  M_{f, g, + 1}(q) = B_{f, g, +1}^+(q) + B_{f, g, +1}^-(q), \quad\quad 
  M_{f, g, - 1}(q) = B_{f, g, -1}^+(q) - B_{f, g, -1}^-(q),
 \end{equation*}
where
\begin{multline}\label{defB}
  B^{\pm}_{f, g, \sigma}(q) = \sum_{\substack{m \equiv \pm n \mods q\\ (mn, q) = 1}}
  \frac{\lambda_f(m) \lambda_g(n)}{(mn)^{1/2}}
  V_{f, g, \sigma}\left(\frac{mn}{q^2}\right) -  \frac{1}{\vphi^{\ast}(q)} \sum_{(m n, q) = 1}
  \frac{\lambda_f(m) \lambda_g(n)}{(mn)^{1/2}}V_{f, g, \sigma
    }\left(\frac{mn}{q^2}\right)
\end{multline}
(indeed, the second term in \eqref{defB} is canceled in the right hand
side of $M_{f, g, -1}(q)$, and for $M_{f, g, +1}(q)$ it compensates
the missing trivial character).
\par
A diagonal main term $\mathrm{MT}_{f,g,\sigma}^d(q)$ is given by the
contribution of $n=m$ in $B_{f, g, \sigma}^+(q)$ (note that $n=m$, $m \equiv -n \not= 0 \, (\text{mod }q)$ is impossible for $q$ odd). By Mellin inversion and a contour shift, we can compute 
explicitly:  
\begin{align*} 
  \mathrm{MT}_{f,g,\sigma}^{d}(q) &=\sum_{\substack{m\geq
      1\\(m,q)=1}}\frac{\lf(m)\lamg(m)}{m}V_{f,g,\sigma
    }\left(\frac{m^2}{q^2}\right)\\
  &=\underset{s=0}{\text{res
    }}\frac{L_\infty(f\otimes\chi,1/2+s)L_\infty(g\otimes\ov\chi,1/2+s)}{L_\infty(f\otimes\chi,1/2)L_\infty(g\otimes\ov\chi,1/2)}\frac{L^{(q)}(f\otimes
    g,1+2s)}{\zeta^{(q)}(2+4s)}\frac{q^{2s}}{s}+O(q^{-1/2+\eps}),
\end{align*}
for any $\eps>0$, where $\chi$ denotes any primitive character of
modulus $q$ of parity $\chi(-1)=\sigma$, $L(f\otimes g, s)$ denotes the
Rankin--Selberg $L$-function of $f$ and $g$, including
\[ L(f\otimes E,s)=L(f,s)^2,\quad L(E\otimes E,s)=\zeta(s)^4, \] and
the superscript $^{(q)}$ denotes omission of the Euler factor at $q$.
\par
Computing the residue explicitly, we find that
\[ \mathrm{MT}_{f,g,\sigma}^{d}(q)=\mathrm{MT}^0_{f,g,\sigma}(q)+O(q^{\eps-1/2}), \]
where
\[ \mathrm{MT}^0_{f,g,\sigma}(q)=\begin{cases}
  P_{1,f,\sigma}(\log q)&\hbox{for $P_{1,f,\sigma}(X)$ a degree 1 polynomial if $f=g$ is cuspidal,}\\
  \frac{L(f\otimes g,1)}{\zeta(2)}&\hbox{if $f\not =g$ are both cuspidal,}\\
  \frac{L(f,1)^2}{\zeta(2)}&\hbox{if $f$ is cuspidal and $g=E$,}\\
  P_{4,\sigma}(\log q)&\hbox{for $P_{4,\sigma}(X)$ a polynomial of degree
    4 if $f=g=E$}. \end{cases} \]
\par
We note also that, by Lemma \ref{propfcteqn}, the root number
$\varepsilon(f, g, \pm 1)$ is always $1$ or always $-1$ if $f$ and $g$
are cuspidal, but it is 1 for exactly one choice of sign if $f$ is
cuspidal and $g = E$. This explains the additional factor of $2$ in
Theorem \ref{cuspthm}.

\def\sumdy{\mathop{\sum\nolimits^{\text{dy}}}}

We apply a partition of unity to the $m,n$ variables and are led to evaluate the dyadic sums
\begin{multline*}\label{partition1}
  \sumdy_{M,N\geq 1}\sum_\stacksum{m\equiv \pm n\mods q}{m\not= n}\frac{\lf(m)\lamg(n)}{(mn)^{1/2}}V_{f,g,\pm 1}\Bigl(\frac{mn}{q^2}\Bigr)W_1\Bigl(\frac{m}M\Bigr)W_2\Bigl(\frac{n}N\Bigr)\\
  - \sumdy_{M,N\geq 1} \frac{1}{q} \sum_{m, n} \frac{\lambda_f(m)
    \lambda_g(n)}{(mn)^{1/2}}V_{f, g, \pm
    1}\left(\frac{mn}{q^2}\right)W_1\Bigl(\frac{m}M\Bigr)W_2\Bigl(\frac{n}N\Bigr),
\end{multline*}
up to an error of size $O(q^{-1+2\theta + \varepsilon})$, for any $\eps>0$, that
arises from removing the condition $(mn, q) = 1$ and replacing
$\vphi^{\ast}(q)$ by $q$. 

In this expression, the symbol $\sumdy$ indicates that $M,N\geq 1$
range over powers of $2$, and $W_1$, $W_2$ are smooth compactly
supported on $[1/2,2]$ satisfying 
$  W_{i}^{(j)}(x)\ll_j 1$
for $i = 1, 2$ and all $j \geq 0$.  Using the rapid decay of
$V_{f,g,\pm 1}(x)$, we may moreover, up to a negligible error term,
assume that $M$, $N$ satisfy
\begin{equation}\label{MN<q2}
  1\leq MN\leq q^{2+\eps}.
\end{equation}
 
In order to evaluate the remaining $O(\log^2 q)$ sums with $M$ and $N$
fixed, we first separate the variables $m$ and $n$. We proceed by
Mellin inversion (as in \cites{MY, BloMil}): using the
definition~(\ref{Vfgdef}) of $V_{f,g,\pm 1}(x)$ as a Mellin transform,
we shift the line of integration to $\Re s=\eps$ and approximate 

\begin{displaymath}
  V_{f,g,\pm1}(x) =\frac{1}{2\pi i} \int_{(\eps),|s|\leq \log^2q} \frac{L_\infty(f\otimes\chi,1/2+s)L_\infty(g\otimes\ov\chi,1/2+s)}{L_\infty(f\otimes\chi,1/2)L_\infty(g\otimes\ov\chi,1/2)} x^{-s} \frac{ds}{s} + O(q^{-100})
\end{displaymath}
due to
the exponential decay of
$
L_\infty(f\otimes\chi,1/2+s)L_\infty(g\otimes\ov\chi,1/2+s)
$
as $|\Im s|\ra\infty$. We exchange summation and integration and, up to
replacing $W_1(x),W_2(x)$ by $x^{-1/2-s}W_1(x),\ x^{-1/2-s}W_2(x)$, we
are led to evaluating bilinear sums of the shape
\begin{multline}\label{defBpmMN}
  B_{f,g}^{\pm}(M,N) =\frac{1}{(MN)^{1/2}}\sum_{\substack{m\equiv\pm n\mods q \\ m \not=  n}}{\lf(m)\lamg(n)}W_1\Bigl(\frac{m}M\Bigr)W_2\Bigl(\frac{n}N\Bigr)\\
  - \frac{1}{q(MN)^{1/2}} \sum_{m, n} \lambda_f(m) \lambda_g(n)
  W_1\Bigl(\frac{m}M\Bigr)W_2\Bigl(\frac{n}N\Bigr),
\end{multline}
with new test functions $W_1$, $W_2$ (which depend on $s$) satisfying
\eqref{Wbound}, since $s=\eps+it$ and $|t|<\log^2 q$.
 
As explained in Section~\ref{sec-outline}, our objective is then to
show that (assuming Conjecture \ref{completesumconj} if both $f$ and $g$ are cuspidal)
\begin{equation}\label{purpose}
B_{f,g}^{\pm}(M,N)=\delta_{f=g=E}\mathrm{MT}^{od,\pm}_{E,E}(M,N)+O( q^{-\eta+\eps})
\end{equation}
for any $\eps>0$, with
\[ \begin{cases}
\eta=1/32, &  f=g=E,\\  \eta=1/68, &   f \text{ cuspidal, } g=E, \\
\eta=\rpfree, & f ,  g  \text{  both cuspidal.}\end{cases} \]
Once this is done (uniformly in terms of $W_1$ and $W_2$), we
can perform the last integration over $s$ and finish the proof of the
theorems.

We will now begin the proof of this estimate.  To ease notation, we
define the exponents $\mu,\nu,\mu^*,\nu^*,\rho$ by
\[ M=q^\mu,\quad N=q^\nu,\quad \mu^*:=2-\mu,\quad \nu^*=2-\nu. \]
By \eqref{MN<q2} we have
\[ 0\leq \mu+\nu\leq 2+\eps. \]
\par
We consider the three main
results in turn.

\subsection{The case \texorpdfstring{$f$ and $g$}{f and g} cuspidal}\label{fgholomorphic} 

Let $\eta=\rpfree$. We prove~\eqref{purpose} subject to Conjecture~\ref{completesumconj}.

By symmetry, we may assume that $0\leq  \mu\leq \nu\leq 2+\eps$ (up to exchanging the roles of $f$ and $g$). We review the various bounds that are available and the ranges of the parameters $\mu,\nu$ for which \eqref{purpose} holds.

\subsubsection*{The trivial bound}
By \eqref{eqtrivial}, we obtain
\eqref{purpose} immediately if $\mu+\nu\leq 2-2\eta-2\theta\nu$.  We can therefore
assume that 
\begin{equation}\label{cortrivialbound}
2-2\eta-2\theta\nu \leq \mu+\nu\leq 2+\eps	
\end{equation}
 and therefore
\begin{equation}\label{muplusnubound}
-2\eta-2\theta\nu\leq \mu-\nu^*\leq \eps.
\end{equation}

\subsubsection*{The shifted convolution bound}
From~\eqref{shiftedfourthbound}, we obtain that \eqref{purpose}
holds unless
\begin{equation}\label{SCPlowerbound}
1-4\eta\leq \nu-\mu\hbox{ or equivalently }\mu+\nu^*\leq 1+4\eta.	
\end{equation}
\par

\subsubsection*{The trivial Voronoi summation bound}
By \eqref{cortrivialbound} and \eqref{SCPlowerbound}, we then have
\[ \nu\geq 3/2-3\eta-2\theta\geq 1+\frac1{1000}, \]
in which case the condition $n \not=m$ is void (since $\mu\leq 1+\eps/2$); it is then
natural to apply the Voronoi summation formula (Lemma \ref{Voronoi})
to the (long) $n$-variable. To this end, we detect the condition
$m\equiv \pm n\mods q$ by additive characters. The trivial character
cancels the second term on the right hand side of \eqref{defBpmMN},
and one obtains the formula
\begin{equation}\label{BMNaftervoronoi}
  B_{f,g}^{\pm}(M,N) = 
  \frac{1}{(qMN^*)^{1/2}}
  \sum_{m,n\geq 1}\lf(m)\lamg(n)W_1\left(\frac{m}M\right)
  \frac1N{\widetilde{W_{2,N}}\left(\frac n{q^2}\right)}\Kl(\pm mn;q) 
\end{equation}
with $N^*=q^2/N$, where we use the notation of Lemma
\ref{besseldecay}.  In particular, by this lemma, the function
\[ y\mapsto \frac{1}N\widetilde{W_{2,N}}\Bigl(\frac y{q^2}\Bigr) \]
decays rapidly for $y\geq q^\eps N^*$ and the contribution to $B_{f,g}^{\pm}(M,N)$ of those $n$ that satisfy  $n\geq q^\eps N^*$ is negligible. By a partition of unity (using Lemma \ref{partition}), we can decompose 
\eqref{BMNaftervoronoi} into a sum of $O(\log q)$ terms of the shape
\[ C^\pm(M,N')=\frac{(1 + N^{\ast}/N')^{2\theta + \varepsilon}}{(qMN^*)^{1/2}}\sum_{m,n\geq 1}\lf(m)\lamg(n)W_1\left(\frac{m}M\right)W_2\left(\frac{n}{N'}\right)\Kl(\pm mn;q), \]
with $W_1,W_2$ satisfying \eqref{Wbound} and $N'=q^{\nu'}\leq q^\eps N^*$.

By Weil's bound for Kloosterman sums $|\Kl(\pm mn;q)|\leq 2$ and \eqref{RP4} we have the trivial bound
\begin{equation}\label{ctriv}
C^\pm (M,N')\ll  q^\eps (MN^*/q)^{1/2}=q^{\eps+\frac{\mu+\nu^*-1}2},
\end{equation}
which establishes \eqref{purpose} unless (cf. \eqref{SCPlowerbound} for the upper bound)
\begin{equation}\label{narrowbound}
1-2\eta \leq \mu+\nu^*\leq 1+4\eta.	
\end{equation}
This together with \eqref{muplusnubound} implies that
$$\mu\geq\frac12-2\eta-\theta\nu.$$

\subsubsection*{The bilinear sum bound}
Now, applying Proposition \ref{bilinearpr} (whose conclusion we recall is
conditional on Conjecture \ref{completesumconj}) with $\mathrm{M}=2M,\ \mathrm{N}=2N^\ast$,
\[ (\alpha_m)_{m\leq 2M}=\lambda_f(m)W_1\left(\frac{m}{M}\right),\ (\beta_n)_{n\leq 2N^*}=\lambda_g(n)W_2\left(\frac{n}{N'}\right), \]
and whose assumptions \eqref{MNbounds2} are satisfied by
\eqref{muplusnubound} and \eqref{narrowbound},
we have (using \eqref{RP4})
\[ C^\pm(M,N')\ll q^{\eps}\Bigl(\frac{MN^*}{q}\Bigr)^{1/2}(M^{-1/2}+q^{11/64}(MN^*)^{-3/16})\ll q^{\eps}(q^{3\eta+\frac12\theta\nu-\frac{1}4}+q^{-\frac{1}{64}+\frac{5}{4}\eta})\ll q^{\eps-\eta} \]
for $\eta=\rpfree$ and $\nu\leq 2+\eps$. This concludes the proof of Theorem \ref{cuspthm}. \qed\\
 
\subsection{The case \texorpdfstring{$f=g=E$}{f=g=E}} Next, we prove Theorem
\ref{improvedyoung}. 

Let $\eta=1/32$.  We are once more in a symmetric case, so we can
assume that $\mu\leq \nu$. Moreover, the Ramanujan--Petersson conjecture is trivially true, so we may apply \eqref{eqtrivial}
to obtain \eqref{purpose} if 
\[ \mu+\nu\leq 2-2\eta,\ 2\eta\leq \nu -\mu. \]
The grouping and the analysis of the main terms was done in \cite{MY},
so we will focus on the error term. Applying first~(\ref{shiftedfourthbound}) we obtain
\[ \ET_{E,E}^{\pm}(M,N)\ll q^{-\eta+\eps}, \]
as desired, unless
\[ \mu+\nu^*\leq 1+4\eta, \]
which we assume from now on.
 
In this  remaining range, the off-diagonal main term $
\mathrm{MT}^{od, \pm}_{E,E}(M,N) \ll q^{-7/16+\eps}
$ is small (cf.\ the second term in Proposition \ref{pr-shifted-trivial}), so that we can assume \eqref{muplusnubound}, and it
suffices to prove the estimate
\begin{equation}\label{purpose2}
  B_{E, E}^{\pm}(M,N)\ll q^{-\eta+\eps}.
\end{equation}

We use the letters $N^*, N'$ etc.\ as in the preceding subsection.  We
detect again the congruence by applying the Voronoi summation formula
to the $n$-variable (note that here we have $\vartheta = 0$ in the notation of Lemma \ref{besseldecay}). This expresses the sum $B_{E,E}^{\pm}(M,N)$ into
a main term and two additional terms. As in \eqref{mainterm}, the
main term is $O(q^{-1+\varepsilon})$, while the error terms decompose
into $O(\log q)$ terms of the shape
\begin{equation}\label{eq-sum}
\frac{1}{(qMN^*)^{1/2}}\sumsum_{m,n}d(m) d(n)
W_1\left(\frac{m}{M}\right)W_2\left(\frac{n}{N'}\right)\Kl(\pm mn;q), 
\end{equation}
where $W_1$, $W_2$ satisfy \eqref{Wbound} (the definition of $W_2$ has
changed from its preceding appearance).
\par
A trivial estimate shows that \eqref{purpose2} holds unless
\begin{equation}\label{munu*bound}
{1-2\eta}\leq \mu+\nu^*\leq {1+4\eta},
\end{equation}
which we then assume.

We further decompose~(\ref{eq-sum}) into $O(\log^4q)$ terms of the
form
\begin{multline}\label{eq-sum-2}
  \frac{1}{(qMN^*)^{1/2}}\sumsum_{m_1,m_2,n_1,n_2} W_1\left(\frac{m_1m_2}{M}\right)W_2\left(\frac{n_1n_2}{N'}\right)\\
  \times
  W\left(\frac{m_1}{M_1}\right)W\left(\frac{m_2}{M_2}\right)W\left(\frac{n_1}{M_3}\right)W\left(\frac{n_2}{M_4}\right)\Kl(\pm
  m_1m_2n_1n_2;q)
\end{multline}
with
\[ M_1M_2= M,\quad M_3M_4= N' \leq N^{\ast}. \]
\par
In~(\ref{eq-sum-2}), we separate the variables $m_1,m_2$
resp. $n_1,n_2$ in $W_1(m_1m_2/M_1M_2)$ and $W_2(n_1n_2/M_3M_4)$ by
inverse Mellin transform: we write
\[ W_1\left(\frac{m_1m_2}{M_1M_2}\right)=\frac1{2\pi i}\int_{(0)}\widehat W_1(s)\frac{M^s_1}{m^s_1}\frac{M^s_2}{m^s_2}ds \]
and exchange the order of summations and integrals.  For any $\eps>0$, the contribution to the
integral of the $s$ such that $|s|\geq q^{\eps}$ is negligible, by
\eqref{Wbound} and repeated integration by parts.
\par
Possibly with different $W_i,\ i=1,2,3,4$, and up to renaming some variables, we are reduced to estimating
sums of the shape
\begin{multline}\label{eq-sum-3}
  S^\pm(M_1,M_2,M_3,M_4)=\frac{1}{(qMN^*)^{1/2}}\sumsum_{m_1,m_2,n_1,n_2}
  W_1\left(\frac{m_1}{M_1}\right)W_2\left(\frac{m_2}{M_2}\right) \\
\times  W_3\left(\frac{m_3}{M_3}\right)W_4\left(\frac{m_4}{M_4}\right)
  \Kl(\pm m_1m_2m_3m_4;q),
\end{multline}
where the $W_i$ satisfy \eqref{Wbound} and hence \eqref{Wbound2}  for $Q=q^{\eps}$, and the $M_i$ written in the shape $M_i=q^{\mu_i}$, $i=1,2,3,4$, satisfy
\[ \mu_1\leq \mu_2\leq \mu_3\leq\mu_4,\quad
\mu_1+\mu_2+\mu_3+\mu_4 = \mu+\nu', \quad \nu' \leq \nu^*. \]
\par

The strategy is the following: if the product of two smooth variables
is long (if $\mu_3+\mu_4$ is large, in particular larger than $3/4$), 
we apply the third part~(\ref{eq-fkm1}) of
Theorem~\ref{CombinedTheorem}, with $MN=M_3M_4$, and we sum trivially
over $m_1$ and $m_2$. If this is not the case, it is possible to factor the
product $m_1m_2n_1n_2$ into a product $mn$ in such a way that an
application of the general bilinear estimate \eqref{typeIIgen} is
beneficial.
 
Explicitly, let $\delta<1/4$ be some parameter such that
\[ 1/4-\delta\leq \frac{1}6(\mu+\nu^*). \]
If
\[ \mu_1+\mu_2\leq \frac14-\delta, \]
we apply~(\ref{eq-fkm1}) with
$MN=M_3M_4$ and sum trivially over $m_1$ and $m_2$, obtaining the bound
\begin{equation}\label{bound1}
  q^{-(A+1)\eps}S^\pm(M_1,M_2,M_3,M_4)\ll q^{\frac{1}2(\mu+\nu^*-\frac{5}4)}+q^{-\frac{\delta}{2}}
\end{equation}
for the constant $A$ occurring in \eqref{eq-fkm1}. 
\par
On the other hand, if
\[ \mu_1+\mu_2\geq \frac14-\delta, \]
then at least one of $\mu_2$ and $\mu_1+\mu_2$ is contained in the interval
\begin{equation}\label{interval}
\Bigl[\frac14-\delta,\frac13(\mu+\nu^*)\Bigr],
\end{equation}
since $\mu_2\leq (\mu+\nu^*)/3$ and $\mu_1\leq\mu_2$. Let $u$ be one of the numbers
$\mu_2$ or $\mu_1+\mu_2$ satisfying this condition. 

We then apply \eqref{typeIIgen} with
\[ (M,N)\mapsfrom(q^u,MN'q^{-u}) \]
there (notice that \eqref{interval} and \eqref{munu*bound} guarantee the assumption $M, N \leq q$ in \eqref{typeIIgen}), and we obtain the bound
\begin{equation*}\label{bound2}
  q^{-\eps}S^\pm(M_1,M_2,M_3,M_4)\ll q^{\frac{1}2(\mu+\nu^*-1-u)}+q^{-\frac{1}{2}(\frac12-u)}\ll
  q^{\frac{1}2(\mu+\nu^*-5/4+\delta)}+q^{\frac16(\mu+\nu^*)-\frac14}.
\end{equation*}

We choose the value of $\delta$ by comparing the second term of the
bound \eqref{bound1} with the first of the bound
\eqref{bound2}. Precisely, we 
take
\[ \delta=\frac{1}2\left(\frac{5}4-(\mu+\nu^*)\right), \]
and therefore we get
\[ q^{-\eps}S^\pm(M_1,M_2,M_3,M_4)\ll q^{\frac{1}4(\mu+\nu^*-\frac54)}+q^{\frac16(\mu+\nu^*)-\frac14} \]
under the assumption $\mu+\nu^*\leq 5/4$. This is indeed valid, by
\eqref{munu*bound}, since $\eta\leq 1/16$. Therefore, by
\eqref{munu*bound}, we find that
\[ q^{-(A+1)\eps}S^\pm(M_1,M_2,M_3,M_4)\ll q^{-\frac{1}{16}+\eta}+q^{-\frac1{12}+\frac23\eta}\ll q^{-\frac1{32}}, \]
as desired. \qed

\begin{remark}\label{rm-young-improved-2}
  The same strategy, but with \eqref{pointwiseSCP} instead of
  \eqref{shiftedfourthbound}, gives a saving of $q^{-1/24}$ if $\theta = 0$ in
  \eqref{pointwiseSCP}.
\end{remark}

\subsection{The mixed case}

We will now  prove Theorem \ref{mixedthm} and  consider the mixed case where $f$ is cuspidal and
$g=E$. Let $\eta=1/68$. 
In the present  case, $M$ and $N$ are not symmetric, and so we will need to
distinguish the cases where $\mu\leq \nu$ and $\mu>\nu$ on several occasions.

Firstly, applying   \eqref{shiftedfourthbound}, we see that \eqref{purpose} holds
unless
\begin{equation}\label{alternative}
|\nu - \mu| \geq 1 - 4\eta,
\end{equation}
which we assume from now on. In particular,  the condition $n \not= m$ is void. 
In order to avoid pathological cases, we derive first a simple, but useful auxiliary bound by applying the Voronoi formula to the longer of the two variables and estimating trivially. We   detect the congruence condition in \eqref{defBpmMN} with additive characters and cancel the contribution of the trivial character with the second term. This gives
\begin{equation}\label{start1}
 B_{f,E}^{\pm}(M,N) =\frac{1}{q(MN)^{1/2}} \sum_{\substack{a\, (\text{mod }q)\\ a \not \equiv 0}}\sum_{m, n}{\lf(m)d(n)}e\left(\frac{a(m \mp  n)}{q}\right)W_1\Bigl(\frac{m}M\Bigr)W_2\Bigl(\frac{n}N\Bigr).
 \end{equation}
If, for instance, $N \geq M$, then applying Lemma \ref{Voronoi} to the $n$-sum yields a ``main term''
\begin{equation}\label{mainterm} 
  \frac{1}{q^2(MN)^{1/2}}
  \Bigl(\int_0^{+\infty}(\log x+2\gamma-2\log q)W_2\Bigl(\frac{x}N\Bigr)dx\Bigr)
  \sum_{m\geq 1}\lf(m)r(m;q)W_1\Bigl(\frac{m}M\Bigr)
\end{equation}
where $r(m;q)=q\delta_{q|m}-1$ is the Ramanujan sum, and  two other terms are of the shape
\begin{equation*} 
\frac{1}{q(MN^*)^{1/2}}\sum_{m,n\geq 1}\lf(m)d(n)W_1\left(\frac{m}M\right)\frac1N{\widetilde{(W_{2,N})}_{\sigma }\left(\frac n{q^2}\right)}S(m, \pm \sigma n; q )
\end{equation*}
with $\sigma \in \{\pm 1\}$ and the notation as in Lemma
\ref{besseldecay}. A similar strategy (without a ``main term'') can be applied if $M > N$. Using Weil's bound for Kloosterman sums and estimating trivially (using \eqref{RP4}), we obtain the bound
\begin{equation}\label{Weil}
  B_{f,E}^{\pm}(M,N) \ll q^{\varepsilon} \left(\frac{q \min(M, N)}{\max(M, N)}\right)^{1/2}. 
\end{equation}
In particular, \eqref{purpose} holds unless
\begin{equation}\label{new-cond}
  |\nu - \mu| \leq 1 + 2\eta, 
\end{equation}
which we assume from now on. We proceed to derive, by various methods depending on whether $M > N$ or $M \leq N$, more elaborate bounds that allow us to treat the range where \eqref{alternative} and \eqref{new-cond} are satisfied.

\subsubsection{The case \texorpdfstring{$M\leq N$}{M<=N}}\label{741} If $N \leq q$, then by \eqref{alternative}, we see that $M \ll q^{4\eta + \varepsilon}$, so that \eqref{eqtrivial} suffices to prove \eqref{purpose}. From now on we assume  $N \geq q$. As the Ramanujan--Petersson conjecture is trivially true for the divisor function, \eqref{eqtrivial} holds with $\theta_g = 0$, and hence  \eqref{muplusnubound} holds in the stronger form
\begin{equation}\label{muplusnubound-new}
-2\eta \leq \mu - \nu^{\ast} \leq \varepsilon. 
\end{equation}

First, we  observe that \eqref{new-cond} and \eqref{muplusnubound-new} imply $\mu \geq 1/2 - 2\eta > 2/5 $, so that the second term in \eqref{defBpmMN} is negligible. In the first term, we open the divisor function, apply smooth partitions of unity and are left with bounding the triple sum
\begin{equation}\label{defC}
C(M, N, N_1, N_2) :=  \frac{1}{\sqrt{MN}} \sum_{n_1n_2 \equiv \pm m \, (\text{mod }q)}  \lambda_f(m) W_1(n_1/N_1) W_2(n_2/N_2) W_3(m/M),  
\end{equation}
where 
\begin{equation}\label{Netc1}
  N_1N_2 = N, \quad N_1 \leq N_2
\end{equation}
and $W_1, W_2, W_3$ are (new) smooth, compactly supported weight functions satisfying \eqref{Wbound}. We can now  apply Proposition \ref{prop-Young}, getting
\begin{equation}\label{bound-1}
C(M, N, N_1, N_2) \ll q^{\varepsilon}  \cdot  \begin{cases} \frac{\sqrt{MN}}{q^{2-\theta}} + \min\left(\frac{(Mq)^{1/2}}{N^{1/2}} + \frac{N_1M^{1/2}}{qN^{1/2}}, \frac{q^{1/4}}{N_1^{1/2}}+ \frac{M^{1/2}}{N_2^{1/2}} + \frac{N_1M^{1/2}}{qN^{1/2}}, \frac{M^{1/2}N_1}{N^{1/2}}\right) ,\\
  \frac{\sqrt{MN}}{q^{2-\theta}}  +  \min\left(\frac{N_1^2}{(MN)^{1/2}}, \frac{ N^{1/6}N_1q^{1/2}}{ M^{2/3}N_2} \right)+ \frac{M^{1/2}}{N^{1/2}} + \frac{M^{1/2}N_1}{q N^{1/2}} + \frac{M^{3/2}}{N_2N^{1/2}} .
 \end{cases}
\end{equation}
The term $\sqrt{MN}/q^{2-\theta} \ll q^{-3/4+\varepsilon}$ is acceptable and can be dropped. 

Alternatively, we can apply Poisson summation to both $n_1, n_2$ (mimicking Voronoi summation on the original $n$-sum). We conclude from \eqref{muplusnubound-new}  and \eqref{alternative} that $\mu \leq 1/2 + 2\eta < 3/5$, so that in particular $(m, q) = (n_1n_2,q) = 1$ in \eqref{defC}. We obtain
\[ C(M, N, N_1, N_2) = \frac{1}{\sqrt{MN}} \frac{N}{q^2} \sum_{m, h_1, h_2} \lambda_f(m) W_3(m/M) W^{\dagger}_1(h_1N_1/q) W_2^{\dagger}(h_2N_2/q) S(\pm m h_1, h_2; q), \]
where $W_1^{\dagger}$ and $W^{\dagger}_2$ denote the Fourier transforms of $W_1$ and $W_2$.  Since $(q, m) = 1$ and the $m$-sum is sufficiently long, the contribution of the terms $q \mid h_1h_2$ is negligible. After a smooth partition of unity,  we are left with $O(\log^2q)$ terms of the form
\[ C'(M, N, N_1, N_2)  := \frac{1}{\sqrt{qMN_1^{\circ}N_2^{\circ}}} \sum_{m, h_1, h_2} \lambda_f(m) W_3(m/M) W_1(h_1/N_1') W_2(h_2/N_2') \Kl(\pm m h_1 h_2; q), \]
where
\begin{equation}\label{Netc}
N_1' \leq N_1^{\circ}, \quad N_2' \leq N_2^{\circ}, \quad N_1^{\circ} = q/N_1, \quad N_2^{\circ} = q/N_2,
\end{equation}
and $W_1, W_2, W_3$ are (new) smooth, compactly supported weight functions satisfying \eqref{Wbound}. 
Notice that $N_1^{\circ} \geq N_2^{\circ}$. We can now use our results on multi-linear forms in Kloosterman sums as developed in Section \ref{exponentialsums}. In particular, we can  apply the bound \eqref{typeIIgen} with $(M,N)\mapsfrom (N_2', MN_1')$ in the notation of Theorem~\ref{CombinedTheorem}, or the bound \eqref{typeIgen} with $(M,N)\mapsfrom   (M N_2', N_1').$ This gives (using \eqref{RP4} several times) 
\begin{equation}\label{51}
C'(M, N, N_1, N_2) \ll q^{\varepsilon} \frac{MN_1'N_2'}{\sqrt{qMN_1^{\circ}N_2^{\circ}}} ((N_2')^{-1/2} + q^{1/4} (MN_1')^{-1/2}), \quad \text{if } MN_1' \leq q, 
\end{equation}
and  
\begin{equation}\label{53}
C'(M, N, N_1, N_2) \ll q^{\varepsilon} \frac{MN_1'N_2'}{\sqrt{qMN_1^{\circ}N_2^{\circ}}}(q^{1/4} (MN_2')^{-1/6} (N_1')^{-5/12}), \quad \text{if } MN_2' \leq (N_1')^2,
\end{equation}
since the condition $MN_1'N_2' \leq q^{3/2}$ and $N_2', MN_2', N_1' \leq q$ are automatic by \eqref{muplusnubound-new},  \eqref{alternative}, \eqref{Netc1} and \eqref{Netc}. 

Combining all estimates we have derived so far, that is \eqref{eqtrivial} with $\theta_g = 0$, \eqref{shiftedfourthbound}, \eqref{Weil}, \eqref{bound-1}, \eqref{51} and \eqref{53},  we need to
find the maximum of the piecewise linear function
\begin{equation*}
\begin{split}
& \textstyle \min \left(\frac{\mu+\nu}{2} - 1, \max\left(\frac{\nu-\mu - 1}{2}, \frac{\nu-\mu - 1}{4}\right), \frac{1+\mu-\nu }{2},  \right. \\
&\textstyle \quad\quad    \max\left(\frac{\mu + 1 - \nu}{2}, \frac{2\nu_1+ \mu - 2 - \nu}{2} \right), \max\left(\frac{1}{4} - \frac{\nu_1}{2}, \frac{\mu - \nu_2}{2}, \frac{2\nu_1 + \mu - \nu - 2}{2}\right), \frac{\mu + 2\nu_1 - \nu}{2},\\
&\textstyle \quad\quad   \max\left(\min\left(2\nu_1 - \frac{\mu+\nu}{2}, \frac{\nu}{6} + \nu_1 + \frac{1}{2} - \nu_2 - \frac{2\mu}{3}\right), \frac{\mu-\nu}{2}, \frac{\mu + 2\nu_1 - 2 - \nu}{2}, \frac{3\mu}{2} - \nu_2 - \frac{\nu}{2}\right),\\
& \textstyle  \quad \quad \left(\frac{2 \mu +2 \nu_1' + 2\nu_2' -1 - \nu_1^{\circ} - \nu_2^{\circ} - \mu }{2} + \max\left( -\frac{\nu_2'}{2}, \frac{1}{4} - \frac{\mu + \nu_1'}{2}\right)\right) \delta_{\mu + \nu_1' \leq 1}, \\
& \textstyle \quad \quad \left. \left(\frac{2 \mu +2 \nu_1' + 2\nu_2' -1 - \nu_1^{\circ} - \nu_2^{\circ} - \mu }{2}  + \frac{1}{4} - \frac{\mu + \nu_2'}{6} - \frac{5\nu_1'}{12}\right)\delta_{ \mu + \nu_2' \leq 2 \nu_1'} \right)
\end{split}
\end{equation*}
subject to the constraints
\[ 0 \leq \mu \leq \nu, \quad \mu + \nu \leq 2, \quad   \nu_1 + \nu_2 = \nu, \quad 0 \leq \nu_1 \leq \nu_2, \quad 0 \leq \nu_1' \leq \nu_1^{\circ} = 1 - \nu_1, \quad 0 \leq \nu_2' \leq \nu_2^{\circ} = 1 - \nu_2. \]
(Of course this expression can be simplified quite a bit.) 
This is a linear optimization problem that can be solved exactly by 
computer in a finite search. One obtains that  the maximum  $-1/68$ is attained at $\mu = 161/306$, $\nu = 449/306$, $(\nu_1, \nu_2) = (9/17, 287/306)$ and (unsurprisingly) $\nu_1' = \nu_1^{\circ}$, $\nu_2' = \nu_2^{\circ}$. The \textsf{Mathematica} code is available after the bibliography.

\subsubsection{The case \texorpdfstring{$M\geq N$}{M>=N}}\label{742} We now assume $\mu \geq \nu$ and observe that \eqref{alternative} and \eqref{new-cond} are still in force. (However, we cannot use \eqref{muplusnubound-new}.) 

In the present case it turns out to be most efficient 
to apply Voronoi summation in \eqref{start1} in both variables. In the critical range this has essentially the effect of switching $N$ and $M$.   The ``main term'' of the $n$-sum is given by \eqref{mainterm} 
and trivially bounded by $O(q^{-1+\varepsilon})$, which is acceptable. 
 Applying Lemma \ref{besseldecay}  and the usual partition of unity to the remaining terms in the Voronoi formula, we are left with bounding
\begin{displaymath}
\begin{split}
\tilde{B}_{f,E}^{\pm}(M,N)  := &\frac{1+(M^{\ast}/M')^{2\theta}}{\sqrt{M^*N^*}}\Bigl| \sum_{m \equiv \pm n \, (\text{mod } q)} \lambda_f(m) d(n) W_1\left(\frac{m}{M'}\right) W_2\left(\frac{n}{N'}\right)\Bigr| \\
&+  \frac{1+(M^{\ast}/M')^{2\theta}}{q\sqrt{M^*N^*}} \Bigl|\sum_{m, n} \lambda_f(m) d(n) W_1\left(\frac{m}{M'}\right) W_2\left(\frac{n}{N'}\right)\Bigr|, 
\end{split}
\end{displaymath}
where
\[ M^{\ast} = \frac{q^2}{M}, \quad  N^{\ast} = \frac{q^2}{N}, \quad M' \ll M^{\ast} q^{\varepsilon}, \quad N' \ll N^{\ast} q^{\varepsilon} \]
and $W_1, W_2$ are new weight functions satisfying \eqref{Wbound}. 
The second term is negligible unless $M' \leq q^{\varepsilon}$, in which case it is trivially bounded by $O(q^{\varepsilon-1} (M/N)^{1/2})$. By \eqref{new-cond}, this is $O(q^{\varepsilon - 1/2 + \eta})$, which is acceptable. For the first term, we can apply Corollary \ref{cor-Young} in addition to the other bounds \eqref{eqtrivial}, \eqref{shiftedfourthbound} and  \eqref{Weil}. This leads to the linear program to maximize
\begin{displaymath}
\begin{split}
& \textstyle \min \left(\frac{\mu+\nu}{2} - 1 + \theta \mu, \max\left(\frac{\mu-\nu - 1}{2}, \frac{\mu-\nu - 1}{4}\right), \frac{1+\nu - \mu }{2},  \right. \\
&\textstyle \quad\quad \theta(\mu^{\ast} - \mu') +  \max\left(\frac{2\mu' - \mu^{\ast} - \nu^{\ast} + 1}{2}, \frac{2\nu_1 +2 \mu' - 2 - \mu^{\ast} - \nu^{\ast}}{2}\right),\\
&  \textstyle  \quad\quad\quad\quad\quad \theta(\mu^{\ast} - \mu') + \max\left(\frac{\mu' + \nu_2 + \frac{1}{2} - \nu^{\ast} - \mu^{\ast}}{2}, \frac{2\mu'  + \nu_1 - \mu^{\ast} - \nu^{\ast}}{2},  \frac{2\nu_1 + 2\mu' - 2 - \mu^{\ast} - \nu^{\ast}}{2}\right), \\
& \textstyle  \quad\quad\quad\quad\quad \theta(\mu^{\ast} - \mu') +  \theta(\mu^{\ast} - \mu') +  \mu'  + \nu_1 - \frac{\mu^{\ast} + \nu^{\ast}}{2},\\
& \textstyle  \quad\quad  \theta(\mu^{\ast} - \mu')   \max\left(\min\left(2\nu_1 - \frac{\mu^{\ast} + \nu^{\ast}}{2}, \frac{2}{3} \nu' + \nu_1 + \frac{1}{2} - \frac{1}{6} \mu' - \nu_2 - \frac{\mu^{\ast} + \nu^{\ast}}{2}\right)\right., \\
& \textstyle \quad\quad\quad\quad\quad\quad\quad\quad \left. \left.    \frac{2\mu' - \mu^{\ast} - \nu^{\ast}}{2}, \frac{2\mu' + 2\nu_1 - 2 - \mu^{\ast} - \nu^{\ast}}{2}, 2\mu' - \nu_2 - \frac{\mu^{\ast} + \nu^{\ast}}{2} \right) \right)
\end{split}
\end{displaymath}
subject to the constraints
\begin{align*}
0 \leq \nu \leq \mu, \quad &\mu + \nu \leq 2, \quad \nu^{\ast} = 2- \nu, \quad \mu^{\ast} = 2-\mu, \\\ & 0 \leq \nu' \leq \nu^{\ast}, \quad 0 \leq \mu' \leq \mu^{\ast}, \quad \nu_1 + \nu_2 = \nu', \quad 0 \leq \nu_1 \leq \nu_2.
\end{align*}
A computer search shows that the maximum in this case is in fact a bit smaller, namely $-1/64$, attained at $\mu = 47/32$, $\nu = 17/32$,  $(\nu_1, \nu_2) = (17/32, 15/16)$ and $\nu' = \nu^{\ast}$, $\mu' = \mu^{\ast}$.  This completes the proof of Theorem \ref{mixedthm}. \qed
 
\begin{remark}\label{remswitch} The reader may wonder why we use the ``switching trick'' at the beginning of Subsection \ref{742} and why the exponents in  Subsection \ref{741} and \ref{742} are different. Young's technique in the version of Proposition \ref{prop-Young} is only efficient if the divisor function is attached to the longer variable, which explains why we need to switch $N$ and $M$ at the beginning of the last subsection. Under this transformation of two applications of the Voronoi summation formula (one for each sum), the range $MN \leq q^2$ becomes $M^{\ast} N^{\ast} \geq q^2$. Of course, we are mostly interested in the case $MN = q^2$ in which case the size conditions are essentially self-dual, but when it comes to optimizing exponents, the ``worst case'' of Subsection \ref{741} satisfies $MN = q^{2-\delta}$ for  $\delta = 1/34$. For the dual problem, however, $M^{\ast} N^{\ast} = q^{2-\delta}$ is forbidden, because we have the general assumption $MN \ll q^{2+o(1)}$, therefore the exponent in Subsection \ref{742} becomes a little bit better. 
\end{remark}

\section{Appendix: Mathematica code} 

\begin{footnotesize}
  
 \subsection*{Section 7.4.1}

$  $\vspace{0.3cm}

\begin{tabular}{ll}
{\tt In[1] := } &   {\tt Maximize[\{Min[(m + n)/2 - 1, 
   Max[(n - m - 1)/2, (n - m - 1)/4],}\\ & {\tt (1 + m - n)/2, 
   Max[(m + 1 - n)/2, (2 n1 + m - 2 - n)/2], }\\ & {\tt
   Max[1/4 - n1/2, (m - n2)/2, (2 n1 + m - n - 2)/2], (m + 2 n1 - n)/
    2, }\\ & {\tt Max[Min[2 n1 - (m + n)/2, 
     n/6 + n1 + 1/2 - n2 - 2 m/3], (m - n)/2, }\\ & {\tt (m + 2 n1 - 2 - n)/2, 
    3 m/2 - n2 - n/2], 
   If[m + n1prime <= 
     1, }\\ & {\tt (2 m + 2 n1prime + 2 n2prime - 1 - n1circ - n2circ - m)/2 
   }\\ & {\tt +  Max[-n2prime/2, 1/4 - (m + n1prime)/2], 10], 
 }\\ & {\tt  If[m + n2prime <= 
     2 n1prime, (2 m + 2 n1prime + 2 n2prime - 1 }\\ & {\tt - n1circ - n2circ - 
        m)/2 + 1/4 - 5 n1prime/12 - (m + n2prime)/6, 10]],}\\ & {\tt m >= 0, 
  n >= m, m + n <= 2, n1 + n2 == n, n1 <= n2, n1 >= 0, }\\ & {\tt  n1prime >= 0, 
  n1prime <= n1circ, n1circ == 1 - n1,  n2prime >= 0, }\\ & {\tt
  n2prime <= n2circ, n2circ == 1 - n2\}, \{m, n, n1, n2, n1prime, }\\ & {\tt
  n2prime, n1circ, n2circ\}]}\\[0.3cm]
 {\tt Out[1] := } & $\Bigl\{ -\frac{1}{68}, \Bigl\{ {\tt m} \rightarrow \frac{161}{306}, \,\,{\tt n} \rightarrow \frac{449}{306}, \,\, {\tt n1} \rightarrow \frac{9}{17}, \,\, {\tt n2} \rightarrow \frac{287}{306}, \,\, {\tt n1prime} \rightarrow \frac{8}{17}, \,\, {\tt n1prime} \rightarrow \frac{19}{306}$,\\ & ${\tt n1circ} \rightarrow \frac{8}{17}, \,\, {\tt n1circ} \rightarrow \frac{19}{306}
  \Bigr\} \Bigr\}$
  \end{tabular} 
 \vspace{0.3cm}  

 \subsection*{Section 7.4.2}

$  $\vspace{0.3cm}

\begin{tabular}{ll}
{\tt In[2] := } &   {\tt Maximize[\{Min[(m + n)/2 - 1 + 7m/64, 
   Max[(m - n - 1)/2, (m - n - 1)/4], }\\ & {\tt (1 + n - m)/2, 
   7/64(mstar-mprime) + Max[(2 mprime - mstar - nstar + 1)/
     2, }\\ & {\tt (2 n1 + 2 mprime - 2 - mstar - nstar)/2], 
   }\\ & {\tt 7/64(mstar-mprime) + Max[(mprime + n2 + 1/2 - nstar - mstar)/
     2, }\\ & {\tt (2 mprime + n1 - mstar - nstar)/
     2, (2 n1 + 2 mprime - 2 - mstar - nstar)/2], }\\ & {\tt 
   7/64(mstar-mprime) + mprime + n1 - (mstar + nstar)/2, }\\
  & {\tt 
   7/64(mstar-mprime) + Max[Min[2 n1 - (mstar + nstar)/2, }\\ & {\tt 
     2/3 nprime + n1 + 1/2 - 1/6 mprime - 
      n2 - (mstar + nstar)/2], }\\ & {\tt (2 mprime - mstar - nstar)/
     2, (2 mprime + 2 n1 - 2 - mstar - nstar)/2, }\\ & {\tt 
    2 mprime - n2 - (mstar + nstar)/2]], 0 <= n, n <= m, m + n <= 2, 
}\\ & {\tt   nstar == 2 - n, mstar == 2 - m, 0 <= nprime, nprime  <= nstar, 
 }\\ & {\tt  0 <= mprime, mprime <= mstar, n1 + n2 == nprime, 0 <= n1, 
  n1 <= n2\},}\\ & {\tt  \{m, n, n1, n2, nprime, nstar, mprime, mstar\}]}\\[0.3cm]
 {\tt Out[2] := } & $\Bigl\{ -\frac{1}{64}, \Bigl\{ {\tt m} \rightarrow \frac{47}{32}, \,\,{\tt n} \rightarrow \frac{17}{32}, \,\, {\tt n1} \rightarrow \frac{17}{32}, \,\, {\tt n2} \rightarrow \frac{15}{32}, \,\, {\tt nprime} \rightarrow \frac{47}{32}, \,\, {\tt nstar} \rightarrow \frac{47}{32}$,\\ & ${\tt mprime} \rightarrow \frac{17}{32}, \,\, {\tt mstar} \rightarrow \frac{17}{32}
  \Bigr\} \Bigr\}$
  \end{tabular} 

\end{footnotesize}

\vspace{1cm}

\end{document}